\documentclass[12pt, a4paper]{article}

\usepackage[latin1]{inputenc}
\usepackage[english]{babel}
\usepackage{amssymb}
\usepackage{amsfonts}
\usepackage{amsmath}
\usepackage{amsthm}
\usepackage{epsfig}
\usepackage{graphics}
\usepackage{dsfont}
\usepackage[usenames, dvipsnames]{xcolor}
\usepackage{verbatim}
\usepackage{latexsym}
\usepackage{setspace}

\usepackage{tikz}
\usepackage{enumerate}
\usepackage{pgfplots}
\usepackage{hyperref}
\usepackage{cancel}

\usepackage{caption} 
\usetikzlibrary{arrows,shapes,trees,patterns}

\theoremstyle{plain}
\newtheorem{theorem}{Theorem}[section]
\newtheorem{corollary}[theorem]{Corollary}
\newtheorem{lemma}[theorem]{Lemma}
\newtheorem{proposition}[theorem]{Proposition}

\theoremstyle{definition}

\theoremstyle{definition}
\newtheorem{remark}[theorem]{Remark}

\newtheorem{example}[theorem]{Example}  
 
\allowdisplaybreaks[4]

\newcommand{\NN}{\mathbb{N}}

\newcommand{\RR}{\mathbb{R}}
\newcommand{\PP}{\mathbb{P}}

\newcommand{\EE}{\mathbb{E}}

\newcommand{\cS}{\mathcal{S}}

\newcommand*\diff{\mathop{}\!\mathrm{d}}

\newcommand{\nrm}[1]{\left\lVert#1\right\rVert}

\newcommand{\bA}{\mathbf{A}}
\newcommand{\ba}{\mathbf{a}}
\newcommand{\bB}{\mathbf{B}}
\newcommand{\bb}{\mathbf{b}}
\newcommand{\bc}{\mathbf{c}}
\newcommand{\bM}{\mathbf{M}}
\newcommand{\bX}{\mathbf{X}}
\newcommand{\bR}{\mathbf{R}}
\newcommand{\bu}{\mathbf{u}}
\newcommand{\bW}{\mathbf{W}}

\newcommand{\RV}{\operatorname{RV}}
\newcommand{\MRV}{\operatorname{MRV}}

\begin{document}
	\title{A 2$\times$2 random switching model and its dual risk model} 
	\author{Anita Behme\thanks{Technische Universit\"at
			Dresden, Institut f\"ur Mathematische Stochastik, Zellescher Weg 12-14, 01069 Dresden, Germany, \texttt{anita.behme@tu-dresden.de} and \texttt{philipp.strietzel@tu-dresden.de}, phone: +49-351-463-32425, fax:  +49-351-463-37251.}\; and Philipp Strietzel$^\ast$}
	\date{\today}
	\maketitle
	
	\vspace{-1cm}
	\begin{abstract}
	In this article a special case of two coupled M/G/1-queues is considered, where two servers are exposed to two types of jobs that are distributed among the servers via a random switch. In this model the asymptotic behavior of the workload buffer exceedance probabilities for the two single servers/ both servers together/ one (unspecified) server is determined. Hereby one has to distinguish between jobs that are either heavy-tailed or light-tailed. The results are derived via the dual risk model of the studied coupled M/G/1-queues for which the asymptotic behavior of different ruin probabilities is determined.
	\end{abstract}
	
	
	

	\section{Introduction}\label{S0}
	\setcounter{equation}{0}
	
	A general $2\times 2$ switch is modelled by a two-server queueing system with two arrival streams. 
	A well-studied special cases of such a switch is given by the \emph{$2\times 2$ clocked buffered switch}, where in a unit time interval each arrival stream can generate only one arrival and each server can serve only one customer; see e.g. \cite{Adan2001,Boxma1993, Cohen1998} and others. This switch is commonly used to model a device used in data-processing networks for routing messages from one node to another.
	
	In this paper we study a $2\times 2$ switch that operates in continuous time, i.e. the arrivals are modelled by two independent compound Poisson processes. Every incoming job is of random size and it is then distributed to the two servers by a random procedure. This leads to a pair of coupled M/G/1-queues. In this model we study the equilibrium probabilities of the resulting workload processes. In particular we determine the asymptotic behavior of the probabilities that the workloads exceed a prespecified buffer. Hereby we will distinguish between workload exceedance of a specific single server, both servers, or one unspecified server. As we will see, the behavior of these workload exceedance probabilities strongly depends on whether jobs are heavy-tailed or light-tailed and we will therefore consider both cases separately.
	
	A related model to the one we study has been introduced in \cite{BoxmaIvanovs} where a pair of coupled queues driven by independent spectrally-positive Lévy processes is introduced. The coupling procedure however is completely different to the switch we shall use. For this model, in \cite{BoxmaIvanovs}, the joint transform of the stationary workload distribution in terms of Wiener-Hopf factors is determined. Two parallel queues are also considered e.g. in \cite{Flatto1984} for an M/M/2-queue where arriving customers simultaneously place two demands handled independently by two servers. We refer to \cite{asmussen} and \cite{mandjes} and references therein for more general information on Lévy-driven queueing systems.
	
	As it is well known, there are several connections between queueing and risk models. In particular the workload
	(or waiting time) in an M/G/1 queue with compound Poisson input is related to the ruin probability in the prominent Cramér-Lundberg risk model, in which the arrival process of claims is defined to be just the same compound Poisson process; see e.g. \cite{asmussen} or \cite{Kyprianou2014}. To be more precise, let 
	$$R(t)=u+ct-\sum_{i=1}^{N(t)} X_i, \quad t\geq 0,$$
	be a \emph{Cramér-Lundberg risk process} with initial capital $u>0$, premium rate $c>0$, i.i.d. claims $\{X_i,i\in\NN\}$ with cdf $F$ such that $X_1>0$ a.s. and $\EE[X_1]=\mu<\infty$, and a claim number process $(N(t))_{t\geq 0}$ which is a Poisson process with rate $\lambda>0$. Then it is well known that the ruin probability 
	$$\Psi(u)=\PP(R(t)<0 \quad \text{for some }t\geq 0)$$
	tends to $0$ as $u\to \infty$, as long as the \emph{net-profit condition}  $\lambda\mu<c$ holds, while otherwise $\Psi(u)\equiv 1$. In particular, if the claims sizes are light-tailed in the sense that an adjustment coefficient $\kappa>0$ exists, i.e.
	$$\exists \kappa>0: \quad \int_0^\infty e^{\kappa x} \overline{F}(x) \diff x = \frac{c}{\lambda},$$
	where $\overline{F}(x)=1-F(x)$ is the tail-function of the claim sizes, then the ruin probability $\Psi(u)$ satisfies the famous \emph{Cramér-Lundberg inequality} (cf.  \cite[Eq. XIII (5.2)]{asmussen}, \cite[Eq. I.(4.7)]{asmussenalbrecher})
	$$\Psi(u)\leq e^{-\kappa u}, \quad u >0.$$
	Furthermore in this case the \emph{Cramér-Lundberg approximation} states that (cf. \cite[Thm. XIII.5.2]{asmussen}, \cite[Eq. I.(4.3)]{asmussenalbrecher})
	$$\lim_{u\to\infty} e^{\kappa u}\Psi(u)=C,$$
	for some known constant $C\geq 0$ depending on the chosen parameters of the model. On the contrary, for heavy-tailed claims
	with a subexponential integrated tail function $\frac{1}{\mu} \int_0^x \overline{F}(y) \diff y$ it is known that (cf. \cite[Thm. X.2.1]{asmussenalbrecher})
	$$\lim_{u\to \infty} \left( \frac{1}{\mu} \int_u^\infty \overline{F}(y) \diff y \right)^{-1}\Psi(u) =\frac{\lambda \mu}{c-\lambda\mu},$$
	and in the special case of tail-functions that are regularly varying this directly implies that the ruin probability  decreases polynomially.\\
	Via the mentioned duality these results can easily be translated into corresponding results on the workload exceedance probability of an M/G/1-queue.	
	
	In this paper we shall use an analogue duality between queueing and risk models in a multi-dimensional setting as it was introduced in \cite{Badila2014}. This allows us to obtain results on the workload exceedance probabilities of the $2\times 2$ switch by studying the corresponding ruin probabilities in the two-dimensional dual risk model.
	
	Bivariate risk models are a well-studied field of research. A prominent model in the literature, that can be interpreted as a special case of the dual risk model in this paper, has been introduced by Avram et. al. \cite{Avram2007}. In this so-called \emph{degenerate model} a single claim process is shared via prespecified proportions between two insurers (see e.g. \cite{Avram2009,Avram2007,Badescu2011, Fossetal, Hu2013}). The model allows for a rescaling of the bivariate process that reduces the complexity to a one-dimensional ruin problem. Exact results and sharp asymptotics for this model have been obtained in  \cite{Avram2009}, where also the asymptotic behavior of ruin probabilities of a general two-dimensional Lévy model under light-tail assumptions is derived.  In \cite{Fossetal} the degenerate model is studied in the presence of heavy tails; specifically asymptotic formulae for the finite time as well as the infinite time ruin probabilities under the assumption of subexponential claims are provided. In \cite{Hu2013} the degenerate model is extended by a constant interest rate. In \cite{Badescu2011} another generalization of the degenerate model is studied that introduces a second source of claims only affecting one insurer. Our risk model defined in Section \ref{S2b} can be seen as a further generalization of the model in \cite{Badescu2011} because of the random sharing of every single claim, compare also with Section \ref{S5b} below.\\
There exist plenty of other papers concerning bivariate risk models of all types and several approaches to tackle the problem. E.g. \cite{Chan2003,Dang2009} consider bivariate risk models of Cram\'er-Lundberg-type with correlated claim-counting processes and derive partial integro-differential equations for infinite-time ruin and survival probabilities in these models.  Various authors focus on finite time ruin probabilities under different assumptions, see e.g. \cite{Chen2013b, Chen2013a, Chen2010, Jiang2015, Li2007, Yang2014, Yuen2006}. E.g. in \cite{Yuen2006} the finite time survival probability is approximated using a so-called bivariate compound binomial model and bounds for the infinite-time ruin probability are obtained using the concept of association.
\\ 
In general dimensions, multivariate ruin is studied e.g. in \cite{Behme2020, Bregman, Cai2005, Gong2012, Konstantinides2016, ramasubramannian}. In particular, in \cite{Behme2020} a bipartite network induces the dependence between the risk processes and this model is in some sense similar to the dual risk model in this paper. Further, in \cite{Collamore1}, multivariate risk processes with heavy-tailed claims are treated and so-called ruin regions are studied, that is,
	sets in $\RR^d$ which are hit by the risk process with small probability. Multivariate regularly varying claims are also
	assumed e.g. in \cite{HultLindskogTR} and \cite{Konstantinides2016}, where in \cite{HultLindskogTR} several lines of business are considered that can balance out ruin, while \cite{Konstantinides2016} focuses exclusively on simultaneous ruin of all business lines/agents. Further, \cite{Samorodnitsky2016} introduces a notion of multivariate subexponentiality and applies this on a multivariate risk process. Note that \cite{HultLindskogTR} and \cite{Samorodnitsky2016} both consider rather general regions of ruin and some of the results from these papers will be applied on our dual risk model.
 
The paper is outlined as follows. In Section \ref{S2} we specify the random switch model that we are interested in and introduce the corresponding dual risk model. Section \ref{S3} is devoted to study both models under the assumption that jobs/claims are heavy-tailed and it is divided into two parts. First, in Section \ref{S3SE} we focus on subexponentiality.  As we shall rely on results from \cite{Samorodnitsky2016} we first concentrate on the risk model in Section \ref{S3SEa} and then transfer our findings to the switch model in Section \ref{S3SEb}. Second, we treat the special case of regular variation in Section \ref{S3RV}, where we start with results for the risk model in Section \ref{S3a}, taking advantage of results given in \cite{HultLindskogTR}, before we transfer our findings to the switch model in Section \ref{S3b}.
	In Section \ref{S4} we assume all jobs/claims to be light-tailed and again first consider the risk model in Section \ref{S4a} before converting the results to the switch context in Section \ref{S4b}. Two particular examples of the switch will then be outlined in Section \ref{S5} where we  also compare the behavior of the exceedance probabilities for different specifications of the random switch via a short simulation study in Section \ref{S5c}. The final Section \ref{S6} collects the proofs of all our findings.

	\section{The switching model and its dual}\label{S2}
	\setcounter{equation}{0}

\subsection{The $2\times 2$ random switching model}\label{S2a}

	Let $\mathcal{W}_1,\mathcal{W}_2$ be servers (or workers) with work speeds $c_1, c_2>0$ and let $\mathcal{J}_1,\mathcal{J}_2$ be two job generating objects. We assume that both objects generate jobs independently with Poisson rates $\lambda_1,\lambda_2>0$, respectively, and that the workloads generated by one object are i.i.d. positive random variables. More specifically, we identify the objects $\mathcal{J}_j$, $j=1,2$, with two independent compound Poisson processes \begin{equation*}
\sum_{k=1}^{N_j(t)} X_{j,k}, \qquad j=1,2 \end{equation*} with jumps $\{X_{j,k}, k\in\mathbb{N}\}$ being i.i.d. copies of two random variables $X_j\sim F_j$ such that $F_j(0)=0$ and $\EE[X_j]<\infty$, $j=1,2$.

The jobs shall be distributed to the two servers by a random switch that is modeled by a random $(2\times 2)$-matrix $\bA=(A_{ij})_{i,j=1,2}$, independent of all other randomness and satisfying the following conditions:
\begin{enumerate}[(i)]
\item $A_{ij}\in[0,1]$ for all $i,j=1,2$, meaning that a job can not be assigned more than totally or less than not at all to a certain server,
\item $\sum_{i=1}^2 A_{ij}=1$ for all $j=1,2$, i.e. every job must be assigned entirely to the servers.
\end{enumerate} 
The switch matrix is triggered independently at every arrival of a job. 

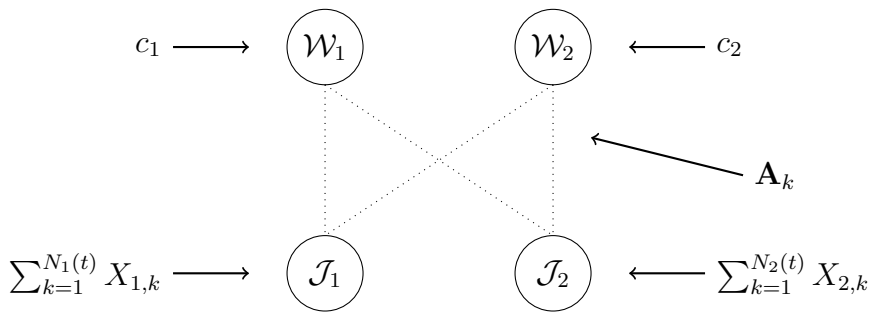
\begin{figure}[!htb]
	\centering
	\begin{tikzpicture}
	\draw (9,3) circle(0.5) node{$\mathcal{W}_1$};
	\draw (12,3) circle(0.5) node{$\mathcal{W}_2$};
	\draw (9,0) circle(0.5) node{$\mathcal{J}_1$};
	\draw (12,0) circle(0.5) node{$\mathcal{J}_2$};
	\draw [style=dotted] (9,2.5) -- (9,0.5);
	\draw [style=dotted] (9,2.5) -- (12,0.5);
	\draw [style=dotted] (12,2.5) -- (9,0.5);
	\draw [style=dotted] (12,2.5) -- (12,0.5);
	\draw [->, thick] (7,3) node[left]{$c_1$} -- (8,3);
	\draw [<-, thick] (13,3) -- (14,3) node[right]{$c_2$};
	\draw [<-, thick](12.5,1.8) -- (14.5,1.3) node[right]{$\mathbf{A}_k$};
	\draw [<-, thick] (8,0) -- (7,0) node[left]{$\sum_{k=1}^{N_1(t)} X_{1,k}$};
	\draw [<-, thick] (13,0) -- (14,0) node[right]{$\sum_{k=1}^{N_2(t)} X_{2,k}$};
	\end{tikzpicture}
	\caption{The random switching model}
	\label{Fig_Model}
\end{figure}

We are interested in the coupled $M/G/1$-queues defined by the resulting storage processes of the two servers, i.e.
\begin{equation}\label{eq-def-W}
W_i(t) = \sum_{j=1}^2\sum_{k=1}^{N_j(t)} (A_{ij})_{k}X_{j,k} - \int_0^t c_i(W_i(s))\diff s \end{equation}
where $\{ \bA_k, k\in\mathbb{N}\}$ are i.i.d. copies of $\bA$ and 
$$c_i(x)=\begin{cases} 0,&  x\leq 0, \\ c_i,& x>0, \end{cases} \quad i=1,2.$$ 
In particular we aim to study the stationary distribution of the multivariate storage process $\bW(t)=(W_1(t), W_2(t))^\top$, that is the distributional limit of $\bW(t)$ as $t\to\infty$ whenever it exists. In this case we write   \begin{equation} \label{eq_def_stationary_dist}  \mathbf{W}:=(W_1,W_2)^\top \end{equation}
for a generic random vector with this steady-state distribution. Note that here and in the following $(\cdot) ^\top$ denotes the  transpose of a vector or matrix.

Let $u>0$ be some fixed buffer barrier for the system and $\mathbf{b}=(b_1,b_2)^\top\in(0,1)^2$
with $b_1+b_2=1$. Set $\bu=\bb u$, i.e. $u_i=b_i u$. Then we are  interested in the \emph{probabilities that the single servers exceed their barriers}
\begin{equation}\label{eq-singleexceed}
\Upsilon_i(u_i)=\mathbb{P}\left(W_i-u_i>0\right), \quad i=1,2, \end{equation}
the \emph{probability that at least one of the workloads exceeds the barrier $u$}
\begin{equation}\label{eq-orexceed}
\Upsilon_\vee(u)=\mathbb{P}\left(\max_{i=1,2} (W_i-u_i)>0\right), \end{equation}
and the \emph{probability that both of the workloads exceed the barrier $u$} \begin{equation} \label{eq-andexceed}
\Upsilon_\wedge(u)=\mathbb{P}\left(\min_{i=1,2} (W_i-u_i)>0\right).\end{equation}

\subsection{The dual risk model}\label{S2b}
	
In the one-dimensional case it is well known that there exists a duality between risk- and queueing models, see e.g. \cite{asmussen}. The multivariate analogue shown in \cite{Badila2014} allows us to formulate the dual risk model to the above introduced random switching model as follows.

Let $N(t):=N_1(t)+N_2(t)$ such that $N(t)$ is a Poisson process with rate $\lambda=\lambda_1+\lambda_2$. Define the multivariate risk process 
\begin{equation} \label{eq-def-R}
\mathbf{R}(t) := \begin{pmatrix} R_1(t) \\ R_2(t) \end{pmatrix} := \sum_{k=1}^{N(t)} \bA_k \bB_k \begin{pmatrix}X_{1,k} \\ X_{2,k} \end{pmatrix}- t\begin{pmatrix} c_1 \\ c_2 \end{pmatrix}=: \sum_{k=1}^{N(t)} \bA_k \bB_k \bX_k - t \bc, \end{equation} 
where $\bB_k$ are i.i.d. random matrices, independent of all other randomness, such that 
\begin{equation*}
\mathbb{P}\left(\bB_k= \begin{pmatrix} 1 & 0 \\ 0 & 0 \end{pmatrix}\right)= \frac{\lambda_1}{\lambda} \qquad \text{ and }\qquad \mathbb{P}\left(\bB_k= \begin{pmatrix} 0 & 0 \\ 0 & 1 \end{pmatrix}\right)= \frac{\lambda_2}{\lambda} \quad \text{for all }k. \end{equation*}

Note that the components of $(\bR(t))_{t\geq 0}$ satisfy the net-profit condition, if 
\begin{equation}\label{eq-safetyloading}
c_i^*:= -\frac{1}{\lambda}\EE[R_i(1)]= \frac{1}{\lambda} \left( c_i -  \lambda_1\mathbb{E}[A_{i1}]\cdot \mathbb{E}[X_1]- \lambda_2 \mathbb{E}[A_{i2}] \cdot\mathbb{E}[X_2] \right)>0 \quad \text{for }i=1,2.
\end{equation}
We will therefore assume \eqref{eq-safetyloading} throughout the paper. Note that as mentioned in \cite{Badila2014}, \eqref{eq-safetyloading} implies existence of the stationary distribution of $\bW(t)$, i.e.  $\bW$ in \eqref{eq_def_stationary_dist} is well-defined. For a proof of this fact in the univariate setting, see e.g. \cite[Thm. 4.10]{Kyprianou2014}. 

For the buffer $u>0$, in the risk model, we define  the \emph{ruin probabilities of the single components}
\begin{equation}\label{eq-singleruin}
\Psi_i(u_i) := \PP(R_i(t) - u_i>0 \text{ for some }t>0), \quad i=1,2,
\end{equation}
the \emph{ruin probability for at least one component}
\begin{equation}\label{eq-orruin}
\Psi_{\vee}(u) := \mathbb{P}\left( \max_{i\in\{1,2\}} \left(R_i(t)-u_i\right)> 0 \text{ for some }t> 0\right),
\end{equation}
and the  \emph{ruin probability for all components}
\begin{equation}\label{eq-andruin}
\Psi_{\wedge}(u) := \mathbb{P}\left( \left(R_i(t_i)-u_i\right)> 0 \text{ for some }t_i> 0, i=1,2\right), \end{equation}
where as before $\bu=\bb u$ for $\mathbf{b}\in(0,1)^2$
with $b_1+b_2=1$.

The following Lemma allows us to gather information about the bivariate storage process in the switching model by performing calculations on our dual risk model.

\begin{lemma} \label{Corollary_Duality}
Consider the distributional limit of the workload process $\bW$ and the risk process $(\bR(t))_{t\geq 0}$ defined in \eqref{eq-def-R} and assume \eqref{eq-safetyloading}. Then the workload exceedance probabilities \eqref{eq-singleexceed}, \eqref{eq-orexceed}, and \eqref{eq-andexceed}, and the ruin probabilities \eqref{eq-singleruin}, \eqref{eq-orruin}, and \eqref{eq-andruin}, fulfil
\begin{align*}
\Upsilon_i(u_i) &= \Psi_i(u_i),\\
\Upsilon_\vee(u) &= \Psi_\vee(u), \\ 
\text{and} \quad \Upsilon_\wedge(u) &= \Psi_\wedge(u), \quad u>0.
\end{align*}
\end{lemma}
\begin{proof}
This follows directly from \cite[Lem. 1]{Badila2014} letting $N\to\infty$ and due to the so-called PASTA property, see \cite[Thm. 6.1]{asmussen}. 
\end{proof}

Note that in the ruin context it is common (see e.g.  \cite{Avram2009} or \cite{Behme2020}) to consider also the \emph{simultaneous ruin probability for all components}
\begin{equation}
\Psi_{\wedge,\text{sim}}(u) := \mathbb{P}\left( \min_{i\in\{1,2\}} \left(R_i(t)-u_i\right)> 0 \text{ for some }t> 0\right). \end{equation}
As we will see, results on $\Psi_{\wedge,\text{sim}}$ can sometimes be shown in analogy to those on $\Psi_\vee$ and we shall do so whenever it seems suitable. However, $\Psi_{\wedge,\text{sim}}$ has no counterpart in the switching model. 

It is clear from the above definitions that for all $\bu=\bb u\in (0,\infty)^2$
\begin{equation}\label{eq-includeexclude2} \Psi_{\wedge,\text{sim}}(u) \leq \Psi_\wedge(u)= \Psi_1(b_1u) + \Psi_2(b_2u) - \Psi_\vee(u),\end{equation}
and likewise
\begin{equation}\label{eq-includeexclude1} \Upsilon_\wedge(u) = \Upsilon_1(b_1 u) + \Upsilon_2(b_2 u) - \Upsilon_\vee(u).\end{equation}
We will therefore focus in our study on $\Upsilon_\vee$ and $\Psi_\vee$ and then derive the corresponding results for $\Upsilon_\wedge$ and $\Psi_\wedge$ via \eqref{eq-includeexclude1} and \eqref{eq-includeexclude2}. 

\subsection{Further notations}

To keep notation short, we write $\RR_{\geq 0}$, and $\RR_{\leq 0}$ for the positive/negative half line of the real numbers, respectively, and likewise use the notations $\RR_{>0}$, and $\RR_{<0}$ such that in particular $\RR_{<0}^2=(-\infty,0)\times(-\infty,0)$. Further $\overline{\RR}=\RR\cup\{-\infty,\infty\}$. For any set $M\subset\RR^q$ we write $\overline{M}$ for its closure, and $\partial M$ for its boundary, i.e. $\overline{M}=M\cup \partial M$.\\
We write $\sim$ for asymptotic equivalence at infinity, i.e. $f\sim g$ if and only if $\lim_{x\to\infty} \frac{f(x)}{g(x)}=1$, while $\nsim$ indicates that such a convergence does not hold. Moreover we use the standard Landau symbols, i.e. $f(x) = o(g(x))$ if and only if $f(x)/g(x) \to 0$ as $x\to \infty$.\\
Lastly, throughout the paper we set $\frac{1}{\infty}:=0$ and $\frac{1}{0}=:\infty$, which yields in particular $\overline{F}(\frac{x}{0}):=0$ for any tail function $\overline{F}$.

\section{The heavy-tailed case} \label{S3}
\setcounter{equation}{0}
	
In this section we will assume that the distribution of the arriving jobs is heavy-tailed. A very general class of heavy-tailed distributions is given by the subexponential distributions, and we will consider this case in Section \ref{S3SE} below. However, as we will see, the asymptotics we obtain in this case are not very explicit, in the sense that in a multivariate setting they do not allow for a direct statement about the speed of decay of the exceedance probabilities. We will therefore proceed and treat the more special case of regularly varying distributions in Section \ref{S3RV}, where speeds of decay can be derived more easily.
	
\subsection{The subexponential case}\label{S3SE}

Recall first that a random variable $X$ in $\RR_{>0}$ with distribution function $F$ is called \emph{subexponential} if 
$$\lim_{x\to \infty} \frac{\overline{F^{\ast 2}}(x)}{\overline{F}(x)} = 2,$$
where $F^{\ast 2}$ is the second convolution power of $F$, i.e. the distribution of $X'+ X''$, where $X'$ and $X''$ are i.i.d. copies of $X$. In this case, we write $F\in \cS$ or $X\in \cS$.

 As we are considering a multivariate setting in this paper, our proofs use a concept of multivariate subexponentiality. Several approaches for this exist and we shall rely here on the definition and results as given in \cite{Samorodnitsky2016}, which also provides a comprehensive  overview of previous notions of multivariate subexponentiality as given in \cite{Cline, Omey2006}.

\subsubsection{Results in the risk context}\label{S3SEa}

We start by presenting our main theorem in the subexponential setting, which we state in terms of the risk process defined in Section \ref{S2b}. Its proof relies on the theory developed in \cite{Samorodnitsky2016} and is given in Section \ref{S3cSE} below.
\begin{theorem} \label{Theorem_Subexponential_Risk_Or}
For all $u>0, v\geq 0$ set 
		\begin{equation}\label{eq_defguv} 
				g(u,v):=\frac{\lambda_1}{\lambda}\cdot \mathbb{E}\left[ \overline{F}_1\left(\min\left\{ \tfrac{u b_1 + vc_1^\ast}{A_{11}}, \tfrac{u b_2 + vc_2^\ast}{A_{21}}\right\}\right)\right]   +  \frac{\lambda_2}{\lambda}\cdot \mathbb{E}\left[\overline{F}_2\left(\min\left\{ \tfrac{u b_1 + vc_1^\ast}{A_{12}}, \tfrac{u b_2 + vc^\ast_2}{A_{22}}\right\}\right)\right], 
		\end{equation}	
		and assume that 
		\begin{equation} \label{eq_theta_result}
			\theta :=  \frac{\lambda_1}{\lambda}\cdot \mathbb{E}[X_1] \cdot \mathbb{E}\left[\min\left\{ \frac{A_{11}}{c_1^*}, \frac{A_{21}}{c_2^*}\right\}  \right] + \frac{\lambda_2}{\lambda}\cdot \mathbb{E}[X_2] \cdot \mathbb{E}\left[\min\left\{ \frac{A_{12}}{c_1^*}, \frac{A_{22}}{c_2^*}\right\}  \right]>0.
		\end{equation}
		Further, define a cdf by 
		\begin{equation}\label{eq_def_cdfF_M}
			F_{\text{subexp}}(u) := 1- \theta^{-1}\int_0^\infty g(u,v)  \diff v, \quad u\geq 0.
		\end{equation}
		and assume that $F_{\text{subexp}} \in\mathcal{S}$, then
		\begin{equation}
			\Psi_\vee(u) \sim \int_0^\infty g(u,v) \diff v, \quad \text{as }u\to\infty.
		\end{equation} 
	\end{theorem}
The asymptotic behavior of the ruin probabilities for single components in the subexponential setting as presented in the next lemma can be shown by classic results. Again a proof is given in Section \ref{S3cSE}.
\begin{lemma}\label{lem-singleruinasymptotic_subexp}
Assume that
\begin{equation*}
F_I^{i}(x) :=  \frac{1}{\lambda_1\cdot \mathbb{E}[A_{i1}]\cdot \mathbb{E}[X_1] + \lambda_2 \cdot \mathbb{E}[A_{i2}]\mathbb{E}[X_2]} \cdot  \mathbb{E}\left[\lambda_1 \int_0^x \overline{F}_1(\tfrac{y}{A_{i1}}) \diff y  + \lambda_2 \int_0^x \overline{F}_2(\tfrac{y}{A_{i2}}) \diff y\right]
\end{equation*} is subexponential. 
Then the ruin probability for a single component \eqref{eq-singleruin} fulfills
\begin{align}\label{eq-singleruinasymp}
\Psi_i(u)
\sim \frac{1}{\lambda} \EE\left[ \int_0^\infty \left(\lambda_1 \overline{F}_1\left(\tfrac{u + v c_i^*}{A_{i1}} \right) + \lambda_2 \overline{F}_2\left(\tfrac{u +vc_i^*}{A_{i2}} \right)\right) \diff v \right],\quad \text{as }u\to\infty.
\end{align}
\end{lemma}
Lastly we consider the joint ruin probability $\Psi_\wedge$ in the following proposition.
\begin{proposition}\label{Proposition_Subexponential_Risk_And}
	Assume that $F_{\text{subexp}}$ as in \eqref{eq_def_cdfF_M}, and $F_I^1,F_I^2$ as in Lemma \ref{lem-singleruinasymptotic_subexp}, are in $\cS$. Recall $\bu=\bb u$ with $b_1+b_2=1$ and $\mathbf{b}=(b_1,b_2)^\top\in(0,1)^2$.  Then if  
	\begin{equation}\label{eq_notAsymptoticEquivalent}
		\Psi_1(b_1u)+\Psi_2(b_2u) \nsim \Psi_\vee(u),
	\end{equation}
	we obtain as $u\to\infty$
	\begin{align*}
	\Psi_{\wedge}(u)
	&\sim  \frac{\lambda_1}{\lambda} \EE\left[ \int_0^\infty  \overline{F}_1\left(\max\left\{ \tfrac{ub_1 + vc_1^*}{A_{11}}, \tfrac{ub_2 + vc_2^*}{A_{21}}\right\}\right)  \diff  v \right] \\
	&\qquad  + \frac{\lambda_2}{\lambda} \EE\left[ \int_0^\infty  \overline{F}_2\left(\max\left\{ \tfrac{ub_1 + vc_1^*}{A_{12}}, \tfrac{ub_2 + vc_2^*}{A_{22}}\right\}\right)  \diff  v \right].
	\end{align*}
	Conversely, if \eqref{eq_notAsymptoticEquivalent} fails, then with $g(u,v)$ as in \eqref{eq_defguv}
	 \begin{equation}  \label{eq-andvanishesSE}
		\Psi_{\wedge}(u) =o\left(\int_0^\infty g(u,v) dv \right), \quad \text{as }u\to\infty.
	\end{equation}
\end{proposition}

\subsubsection{Results in the switch context}\label{S3SEb}
With the help of Lemma \ref{Corollary_Duality} we may now directly summarize our findings from the last section to provide a rather explicit insight into the asymptotic behavior of the workload barrier exceedance probabilities in the switching model defined in Section \ref{S2a}.
\begin{corollary}[Asymptotics of the exceedance probabilities under subexponentiality] \label{Corollary_Subexp_Asymptotics}
Assume that $F_{\text{subexp}}$ as in \eqref{eq_def_cdfF_M}, and $F_I^1,F_I^2$ as in Lemma \ref{lem-singleruinasymptotic_subexp}, are in $\cS$. 
Define the resulting integrated tail functions for servers $i=1,2$ via
\begin{equation} \label{eq_definition_mixed_FI} \overline{F}_{I,i}(u,\bA):=
 \lambda_1  \int_0^\infty \overline{F}_1\left(\tfrac{u + v c_i^*}{A_{i1}} \right) \diff v + \lambda_2 \int_0^\infty \overline{F}_2\left(\tfrac{u +vc_i^*}{A_{i2}} \right) \diff v, \quad u>0.
\end{equation}
Then the workload exceedance probabilities  \eqref{eq-singleexceed}, \eqref{eq-orexceed}, and \eqref{eq-andexceed} fulfil
\begin{align*}
\Upsilon_i(b_i u)
&\sim \frac{1}{\lambda}\cdot \EE\left[\overline{F}_{I,i}(b_i u,\bA)\right] , \quad i=1,2,\\
\label{eq_OrAsymptotics_Subexp}
\Upsilon_\vee(u) &\sim \frac{\lambda_1}{\lambda} \EE \left[   \int_0^\infty \overline{F}_1\left(\min\left\{ \tfrac{ub_1 + vc_1^*}{A_{11}}, \tfrac{ub_2 + vc_2^*}{A_{21}}\right\}\right)\diff v \right] \\
&\qquad  +  \frac{\lambda_2}{\lambda} \EE \left[  \int_0^\infty \overline{F}_2\left(\min\left\{ \tfrac{ub_1 + vc_1^*}{A_{12}}, \tfrac{ub_2 + vc_2^*}{A_{22}}\right\}\right) \diff v \right]
\end{align*}
and, assuming additionally that 
\begin{equation}\label{eq_notAsymptoticEquivalentSW}
\Upsilon_1(b_1u)+\Upsilon_2(b_2u) \nsim \Upsilon_\vee(u),
\end{equation}
we obtain 
 \begin{align*}
\Upsilon_\wedge(u) &\sim  \frac{\lambda_1}{\lambda} \EE\left[ \int_0^\infty  \overline{F}_1\left(\max\left\{ \tfrac{ub_1 + vc_1^*}{A_{11}}, \tfrac{ub_2 + vc_2^*}{A_{21}}\right\}\right)  \diff  v \right] \\ & \qquad + \frac{\lambda_2}{\lambda} \EE\left[ \int_0^\infty  \overline{F}_2\left(\max\left\{ \tfrac{ub_1 + vc_1^*}{A_{12}}, \tfrac{ub_2 + vc_2^*}{A_{22}}\right\}\right)  \diff  v \right].
\end{align*}
If \eqref{eq_notAsymptoticEquivalentSW} fails, then \begin{equation*}  \Upsilon_{\wedge}(u) =o\left( \mathbb{E}\left[\overline{F}_{I,1}(b_1 u,\bA)+\overline{F}_{I,2}(b_2 u,\bA)\right] \right). \end{equation*}
\end{corollary}
\begin{proof}
	This is clear from Lemma \ref{Corollary_Duality}, Theorem \ref{Theorem_Subexponential_Risk_Or}, Lemma \ref{lem-singleruinasymptotic_subexp}, and Proposition \ref{Proposition_Subexponential_Risk_And}.
\end{proof}

\subsection{The regularly varying case}\label{S3RV}

In this section we will restrict the class of considered heavy-tailed distributions and assume that the tail functions of the arriving jobs are regularly varying. As we will see, this restriction leads to a much more explicit description of the asymptotic behavior of ruin and exceedance probabilities. \\
Let $f\colon \mathbb{R}\to (0,\infty)$ be a measurable function and recall that $f$ is \emph{regularly varying (at infinity) with index $\alpha\geq 0$} if for all $\lambda>0$ it holds that 
\begin{equation*}
\lim_{t\to\infty} \frac{f(\lambda t)}{f(t)} =\lambda^{\alpha},
\end{equation*}
with the case $\alpha=0$ typically being referred to as \emph{slowly varying}. In this case we write $f\in\operatorname{RV}(\alpha)$. 
A real-valued random variable $X$ is called \emph{regularly varying with index $\alpha\geq0$}, i.e. $X\in \RV(\alpha)$, if its tail function $\overline{F}(\cdot)=\PP(X>\cdot)$ is regularly varying with index $-\alpha$. It is well-known that $\RV(\alpha)\subset \cS$ for all $\alpha\geq 0$, cf. \cite[Prop. X.1.4]{asmussenalbrecher}.

Further we follow \cite{HultLindskogTR} and call a random vector $\mathbf{Z}$ on $\mathbb{R}^q$ \emph{multivariate regularly varying} if there exists a non-null measure $\mu$ on $\overline{\mathbb{R}}^q\backslash\{\mathbf{0}\}$ such that
\begin{enumerate}[(i)]
	\item $\mu\left(\overline{\mathbb{R}}^q\backslash\mathbb{R}^q\right)=0$,
	\item $\mu(M)<\infty$ for all Borel sets $M$ bounded away from $\mathbf{0}$,
	\item for all Borel sets $M$ satisfying $\mu(\partial M)=0$ it holds that \begin{equation} \label{eq-vectorregvarying}
	\frac{\mathbb{P}(\mathbf{Z}\in tM)}{\mathbb{P}(\nrm{\mathbf{Z}}>t)} \to \mu(M). \end{equation}
\end{enumerate}
The norm $\|\cdot \|$ will typically be chosen to be the $L^1$-norm in this article.
If $\mathbf{Z}$ is multivariate regularly varying, necessarily there exists $\alpha>0$ satifsfying that for all $M$ as in \eqref{eq-vectorregvarying} and $t>0$ it holds that 
\begin{equation*}
\mu(tM) = t^{-\alpha}\mu(M). 
\end{equation*}
Thus we write $\mathbf{Z}\in \operatorname{MRV}(\alpha,\mu)$.\\ Note that in the one-dimensional case the above definitions coincide. We refer to \cite{Basrack2000} and \cite{resnick} for references of the above and more detailed information on multivariate regular variation. 

\subsubsection{Results in the risk context}\label{S3a}

We will now present our first main result in the regularly varying context. Note that this is not obtained by an application of our above results in the special case of regular variation, but instead we give an independent proof of Theorem \ref{Theorem_Asymptotics} in Section \ref{S3c} that relies on results from \cite{Hult2005}. This approach also allows to consider the simultaneous ruin probability, which had not been possible with the methods used in Section \ref{S3SE} due to stronger assumptions on the involved ruin sets.

\begin{theorem}\label{Theorem_Asymptotics}
	Assume the claim size variables $X_1,X_2$ are regularly varying, i.e. $X_1\in\operatorname{RV}(\alpha_1)$, and  $X_2 \in\operatorname{RV}(\alpha_2)$ for $\alpha_1,\alpha_2>1$.
	Then with $\bA$, $\bB$ from Section \ref{S2} and $\bX=(X_1,X_2)^\top$ it follows that there exists a measure $\mu^\ast$ such that 
	$$\bA\bB\bX \in \MRV(\min\{\alpha_1,\alpha_2\},\mu^\ast).$$
	Further
	\begin{equation} \label{eq_Asymptotics_1}
	\lim_{u\to\infty} \frac{\Psi_{\vee}(u)}{u\cdot\mathbb{P}(\nrm{\bA\bB\mathbf{X}}>u)} = \int_0^\infty \mu^*(v\mathbf{c}^*+\mathbf{b}+ L_{\vee})\diff v =: C_{\vee}<\infty,
	\end{equation} 
	and
	\begin{equation} \label{eq_Asymptotics_2}
	\lim_{u\to\infty} \frac{\Psi_{\wedge,\text{sim}}(u)}{u\cdot\mathbb{P}(\nrm{\bA\bB\mathbf{X}}>u)} = \int_0^\infty \mu^*(v\mathbf{c}^*+\mathbf{b}+ L_{\wedge,\text{sim}})\diff v =: C_{\wedge,\text{sim}}<\infty,
	\end{equation}
	with $\bc^\ast = (c_1^\ast,c_2^\ast)^\top$, $L_\vee = \RR^2\backslash \RR_{\leq 0}^2,$ and $L_{\wedge,\text{sim}} = \RR_{>0}^2$.
\end{theorem}
 
 Note that by conditioning on $\bB$ we have
 \begin{align}
 \mathbb{P}(\nrm{\bA\bB\mathbf{X}}>u) &= \frac{\lambda_1}{\lambda}\cdot  \mathbb{P}\left(\nrm{\begin{pmatrix} A_{11}X_1 \\ A_{21}X_1\end{pmatrix}}>u\right) +\frac{\lambda_2}{\lambda} \cdot \mathbb{P}\left(\nrm{\begin{pmatrix} A_{12}X_2 \\ A_{22}X_2\end{pmatrix}}>u\right) \nonumber \\
 &= \lambda^{-1} \left(\lambda_1 \overline{F}_1(u) + \lambda_2\overline{F}_2(u)\right). \label{eq-Nenner}
 \end{align}
 
 Using the limiting-measure property of $\mu^\ast$ it is further possible to explicitely compute the constants $C_\vee$, and $C_{\wedge,\text{sim}}$ in Theorem \ref{Theorem_Asymptotics} above. This then yields the following proposition whose proof is also postponed to Section \ref{S3c}.

\begin{proposition}\label{Cor-AsymptoticsRuin}
	Assume $X_1\in\operatorname{RV}(\alpha_1)$, and  $X_2 \in\operatorname{RV}(\alpha_2)$ for $\alpha_1,\alpha_2>1$ and set 
	\begin{equation*}
	\zeta:= \lim_{t\to\infty} \frac{\lambda_1 \overline{F}_1(t)}{\lambda_2 \overline{F}_2(t)} \in[0,\infty], 
	\end{equation*}
	such that clearly $\zeta\in(0,\infty)$ implies $\alpha_1=\alpha_2$. Then 
	\begin{align} \label{eq-orasymptotic}
	\lefteqn{  \Psi_{\vee}(u) \sim \frac{C_\vee}{\lambda} \cdot u \left(\lambda_1\overline{F}_1(u) + \lambda_2 \overline{F}_2(u)\right)  \qquad \text{with}}\\
		C_\vee&:=  \mathbb{E}\left[\int_0^\infty  \frac{\zeta\cdot \left(\min\left\{\frac{vc_1^*+b_1}{A_{11}},\frac{vc_2^*+b_2}{A_{21}}\right\}\right)^{-\alpha_1} +\left(\min\left\{\frac{vc_1^*+b_1}{A_{12}},\frac{vc_2^*+b_2}{A_{22}}\right\}\right)^{-\alpha_2}}{1+\zeta} \diff v\right], \label{eq_defcvee} 
	\end{align}
	and
	\begin{align} \label{eq-andsimasymptotics}
	\lefteqn{ \Psi_{\wedge,\text{sim}}(u) \sim \frac{C_{\wedge,\text{sim}}}{\lambda}\cdot u \left(\lambda_1\overline{F}_1(u) + \lambda_2 \overline{F}_2(u)\right) \qquad \text{with}} \\
	C_{\wedge,sim}&:= \mathbb{E}\left[\int_0^\infty  \frac{\zeta \cdot \left(\max\left\{\frac{vc_1^*+b_1}{A_{11}},\frac{vc_2^*+b_2}{A_{21}}\right\}\right)^{-\alpha_1}+\left(\max\left\{\frac{vc_1^*+b_1}{A_{12}},\frac{vc_2^*+b_2}{A_{22}}\right\}\right)^{-\alpha_2}}{1+\zeta}\diff v\right], \nonumber
	\end{align}
	where we interpret $\frac{\infty\cdot x}{\infty}:=x$.
\end{proposition}

We continue our study of the asymptotics of the risk model by determining the asymptotic behavior of $\Psi_\wedge$. 
It is clear from Equations \eqref{eq-includeexclude2} and \eqref{eq-orasymptotic} that in order to do this, we first have to determine the asymptotic behavior of the ruin probabilities for single components \eqref{eq-singleruin}, 
which will be given by the following lemma.

\begin{lemma}\label{lem-singleruinasymptotic}
	Assume $X_1\in\operatorname{RV}(\alpha_1)$, and  $X_2 \in\operatorname{RV}(\alpha_2)$ for $\alpha_1,\alpha_2>1$. 
	Then the ruin probability for a single component \eqref{eq-singleruin} fulfills \eqref{eq-singleruinasymp}.
\end{lemma}

With this the following proposition is straightforward. Again, the proof can be found in Section \ref{S3c}.

\begin{proposition}\label{prop-andasympruin}
Assume $X_1\in\operatorname{RV}(\alpha_1)$, and  $X_2 \in\operatorname{RV}(\alpha_2)$ for $\alpha_1,\alpha_2>1$.
 Recall $\bu=\bb u$ with $b_1+b_2=1$ and $\mathbf{b}=(b_1,b_2)^\top\in(0,1)^2$.
Then if \eqref{eq_notAsymptoticEquivalent} holds,  
\begin{align}\label{eq-andasymptotic}
\Psi_{\wedge}(u)
&\sim \frac{1}{\lambda}\cdot\left(\lambda_1 \left( \EE\left[\overline{F}_{1,I}(u,\bA)  \right] - C_\vee u \overline{F}_1(u) \right)  
+ \lambda_2 \left( \EE\left[ \overline{F}_{2,I}(u,\bA)\right]  - C_\vee u \overline{F}_2(u) \right)\right),
\end{align}
with $C_\vee$ as defined in \eqref{eq_defcvee} and with the weighted integrated tail functions
\begin{align*}
\overline{F}_{j,I}(u,\bA) &:=  \int_{0}^\infty \overline{F}_j \left(\frac{u b_1 + v c_1^\ast }{A_{1j}}\right) \diff v+ \int_{0}^\infty \overline{F}_j \left(\frac{u b_2+ vc_2^\ast}{A_{2j}}\right) \diff v, \quad u>0, j=1,2.
\end{align*}
 Otherwise, if \eqref{eq_notAsymptoticEquivalent} fails, then
\begin{equation}
\label{eq-andvanishes}
\Psi_\wedge(u) = \mathit{o}\left(u\cdot\left( \overline{F}_1(u) + \overline{F}_2(u)\right)\right).
\end{equation}
\end{proposition}

\subsubsection{Results in the switch context}\label{S3b}

Again we may now summarize our findings in the context of the switching model defined in Section \ref{S2a} as follows.

\begin{corollary}[Asymptotics of the exceedance probabilities for regularly varying jobs] \label{Theorem_RV_Asymptotics}
Assume the workload variables $X_1,X_2$ are regularly varying, i.e. $X_1\in\operatorname{RV}(\alpha_1)$, and  $X_2 \in\operatorname{RV}(\alpha_2)$ for $\alpha_1,\alpha_2>1$. 
Set 
\begin{equation*}
\zeta:= \lim_{t\to\infty} \frac{\lambda_1 \overline{F}_1(t)}{\lambda_2 \overline{F}_2(t)} \in[0,\infty], 
\end{equation*}
such that $\zeta\in(0,\infty)$ implies $\alpha_1=\alpha_2$. Recall $C_\vee$ from \eqref{eq_defcvee} and the integrated tail functions for servers $i=1,2$ from \eqref{eq_definition_mixed_FI}.
Then the workload exceedance probabilities  \eqref{eq-singleexceed} and \eqref{eq-orexceed} fulfil
\begin{align*}
\Upsilon_i(b_i u)
&\sim \frac{1}{\lambda}\cdot \EE\left[\overline{F}_{I,i}(b_i u,\bA)\right] , \quad i=1,2,\\
\label{eq_OrAsymptotics}
\Upsilon_{\vee}(u)& \sim \frac{C_\vee}{\lambda}\cdot u (\lambda_1\overline{F}_1(u) + \lambda_2 \overline{F}_2(u)),
\end{align*}
Assuming additionally \eqref{eq_notAsymptoticEquivalentSW}
the workload exceedance probability \eqref{eq-andexceed} fulfills
 \begin{align*}
 \Upsilon_{\wedge}(u)
&\sim \frac{1}{\lambda}\cdot \left(\EE\left[ \overline{F}_{I,1}(b_1 u,\bA) \right] +  \EE\left[\overline{F}_{I,2}(b_2 u,\bA) \right]- \left(\lambda_1  u\overline{F}_1(u) +  \lambda_2 u\overline{F}_2(u)\right)C_\vee\right).
\end{align*}
If \eqref{eq_notAsymptoticEquivalentSW} fails, then \begin{equation*} 
\Upsilon_{\wedge}(u) =o\left(u\cdot\left(\overline{F_1}(u) + \overline{F_2}(u)\right) \right). \end{equation*}
\end{corollary}
\begin{proof}
	This is clear from Lemma \ref{Corollary_Duality}, Lemma \ref{lem-singleruinasymptotic}, and Propositions \ref{Cor-AsymptoticsRuin} and \ref{prop-andasympruin}.
\end{proof}
\begin{remark}
		At first sight the structure of the asymptotic formulae for $\Psi_{\vee}$ and $\Upsilon_{\vee}$ in the regularly varying and the subexponential case looks pretty similar, as both formulae rely on a minimum inside an integral. However, in the case of regularly varying claims the formulae immediately provide the principal behavior of the tail, only the constant needs more computation. On the contrary, in the subexponential case the initial capital $\mathbf{u}$ is involved strongly inside the integral and even to obtain the asymptotics up to a constant, one has to calculate the integral explicitely.
	\end{remark}
\begin{example} \label{Ex_SwitchDifferentAlphas}
In the setting of Corollary \ref{Theorem_RV_Asymptotics} assume that $\alpha_1<\alpha_2$. 
Then in all asymptotics given in Corollary \ref{Theorem_RV_Asymptotics} the terms including $F_2$ that are regularly varying with index $-\alpha_2+1$ are dominated by the terms involving $F_1$ which are regularly varying with index $-\alpha_1+1$. This yields that in this case
\begin{align*} 
\lim_{u\to\infty} \frac{\Upsilon_i(b_i u)}{\mathbb{E}\left[ \int_{0}^\infty \overline{F}_1\left(\tfrac{u b_i + v c_i^\ast}{A_{i1}}\right)\diff v\right]} &= \frac{\lambda_1}{\lambda}, \quad  i=1,2, 
\end{align*}
as long as $\PP(A_{i1}=0)<1$. Similarly, since $\zeta=\infty$, we obtain
\begin{align*}
\lim_{u\to\infty} \frac{\Upsilon_{\vee}(u)}{u\cdot \overline{F}_1(u)}& =\frac{\lambda_1 C_\vee}{\lambda}=\frac{\lambda_1}{\lambda}  \mathbb{E}\left[\int_0^\infty  \left(\min\left\{\tfrac{vc_1^*+b_1}{A_{11}},\tfrac{vc_2^*+b_2}{1-A_{11}}\right\}\right)^{-\alpha_1} \diff v\right].
\end{align*}
With these observations at hand we may conclude that \eqref{eq_notAsymptoticEquivalentSW} holds if and only if 
\begin{equation} \label{eq_notAsymptoticEquivalent_SC}
\lim_{u\to\infty} \frac{\mathbb{E}\left[\int_{0}^\infty \overline{F}_1\left(\tfrac{u b_1 + v c_1^\ast}{A_{11}}\right)\diff v + \int_{0}^\infty \overline{F}_1\left(\tfrac{u b_2 + v c_2^\ast}{1-A_{11}} \right)\diff v \right]}{\mathbb{E}\left[\int_0^\infty  \left(\min\left\{\frac{vc_1^*+b_1}{A_{11}},\frac{vc_2^*+b_2}{1-A_{11}}\right\}\right)^{-\alpha_1} \diff v\right] u \cdot \overline{F}_1(u) }\neq 1.
\end{equation} 
Thus, given \eqref{eq_notAsymptoticEquivalent_SC}, we get 
\begin{equation*}
\lim_{u\to\infty} \frac{\Upsilon_{\wedge}(u)}{\mathbb{E}\left[\int_{0}^\infty \overline{F}_1\left(\tfrac{u b_1 + v c_1^\ast}{A_{11}}\right)\diff v + \int_{0}^\infty \overline{F}_1\left(\tfrac{u b_2 + v c_2^\ast}{1-A_{11}} \right)\diff v \right]- C_\vee  u \cdot\overline{F}_1(u)}
= \frac{\lambda_1}{\lambda},
\end{equation*}
while otherwise
\begin{equation*}
 \Upsilon_{\wedge}(u) =o\left(u\cdot\overline{F_1}(u)\right). \end{equation*}
\end{example}

\begin{remark}
	The above example can be generalized in the sense that a regularly varying tail dominates any lighter tail, no matter whether this is regularly varying as well or not.\\
	Indeed, assuming that w.l.o.g. $X_1\in \operatorname{RV}(\alpha)$ for $\alpha>1$ and $X_2$ is such that \begin{equation} \label{eq_condition_mixed}
	\overline{F}_2(x) = o(\overline{F}_1(x))
	\end{equation}
	one can prove in complete analogy to the results from the last subsection, that the workload exceedance probabilities  \eqref{eq-singleexceed}, \eqref{eq-orexceed}, and \eqref{eq-andexceed}  fulfil
\begin{align*}
	 \Upsilon_i(b_iu) &\sim 
	 \frac{\lambda_1}{\lambda } \mathbb{E}\left[\int_{0}^\infty \overline{F}_1\left(\tfrac{b_i u + vc_i^\ast}{A_{i1}}\right) \diff v\right]  , \quad i=1,2,\\
	 \Upsilon_{\vee}(u) &\sim \frac{\lambda_1}{\lambda} \mathbb{E}\left[\int_0^\infty \Big(\min\Big\{\tfrac{vc_1^*+b_1}{A_{11}},\tfrac{vc_2^*+b_2}{A_{21}}\Big\}\Big)^{-\alpha}\diff v\right]\cdot u \overline{F}_1(u),
	\end{align*}
	and, assuming additionally that \eqref{eq_notAsymptoticEquivalentSW} holds,
\begin{align*}
	\Upsilon_{\wedge}(u) &\sim  \frac{\lambda_1}{\lambda}\left( \mathbb{E}\left[\int_{0}^\infty \overline{F}_1\left(\tfrac{b_1 u + vc_1^\ast}{A_{11}}\right) \diff v \right] +  \mathbb{E}\left[ \int_{0}^\infty \overline{F}_1\left(\tfrac{b_2 u + vc_2^\ast}{A_{21}}\right) \diff v \right] \right.\\
	& \qquad \qquad \left.-  \mathbb{E}\left[\int_0^\infty \Big(\min\Big\{\tfrac{vc_1^*+b_1}{A_{11}},\tfrac{vc_2^*+b_2}{A_{21}}\Big\}\Big)^{-\alpha}\diff v\right] \cdot u\overline{F}_1(u)\right),
	\end{align*}
 while otherwise \begin{equation*}
	\Upsilon_{\wedge}(u) =o\left(u\cdot \overline{F_1}(u) \right). \end{equation*}	\end{remark}

	\section{The light-tailed case}\label{S4}
	\setcounter{equation}{0}

 In this section we will study the asymptotic behavior of ruin/workload exceedance probabilities for claims/jobs that are typically small, i.e. we will assume throughout this section that the moment generating functions $\varphi_{X_j}(x)=\EE[\exp(xX_j)]$, $j=1,2$, are such that  \begin{equation}\label{eq-lightcondition}
\varphi_{X_j}(x_j)<\infty\;\text{ for some } x_j>0, \quad j=1,2.
\end{equation}

\subsection{Results in the risk context}\label{S4a}

As in the heavy-tailed setting we start by studying the dual risk model. Again, the ruin probabilities for the single components are particularly easy to treat. The following lemma is obtained by a direct application of Lundberg's well-known inequality and the Cramér-Lundberg approximation, see e.g. \cite[Thms. IV.5.2 and IV.5.3]{asmussenalbrecher}. In Section \ref{S4c} a short proof is provided.

\begin{lemma}\label{lem-singleruinlight}
	Assume the claim size variables $X_1,X_2$ fulfill\eqref{eq-lightcondition} and assume there exist (unique) solutions $\kappa_1,\kappa_2>0$ to
	\begin{align} \label{eq_ThLightCond}
	c_i\kappa_i &= \mathbb{E}\left[ \lambda_1(\varphi_{X_1}(\kappa_i A_{i1})  -1) +  \lambda_2 (\varphi_{X_2}(\kappa_i A_{i2})  -1)\right], \quad i=1,2.
	\end{align}
Then the ruin probabilities of the single components fulfil
	\begin{align*}
	\Psi_i(u) \leq e^{-\kappa_i u} \;\text{for all }u>0,\quad \text{and} \quad  \Psi_i(u) \sim C_i e^{-\kappa_i u}, \quad i=1,2,
	\end{align*}
	where 
	\begin{align} 
	C_i & = \frac{\lambda c_i^\ast}{ \EE[\lambda_1 A_{i1}\varphi_{X_1}'(\kappa_iA_{i1})+ \lambda_2A_{i2}\varphi_{X_2}'(\kappa_iA_{i2})] -c_i}\label{eq_LightConstants}  , \quad i=1,2.
	\end{align}
\end{lemma}

Using \eqref{eq-includeexclude2} in the form $\Psi_\vee(u)\leq \Psi_1(b_1 u)+\Psi_2(b_2 u)$ we easily derive the following Lundberg-type bound for $\Psi_\vee$ from the above Lemma.

\begin{corollary}
	Assume the claim size variables $X_1,X_2$ fulfill\eqref{eq-lightcondition} and assume there exist (unique) solutions $\kappa_1,\kappa_2>0$ to \eqref{eq_ThLightCond}, then the ruin probability for at least one component fulfills
	$$\Psi_\vee(u)\leq (e^{-\kappa_1 b_1 u}+e^{-\kappa_2 b_2 u})\wedge 1 \quad \text{for all }u>0.$$
\end{corollary}

\begin{remark}
	Similarly to what has been done in \cite[Thm. 6.1]{Behme2020} it is also possible to derive a Lundberg bound for $\Psi_{\wedge,\text{sim}}$ via classical martingale techniques. Indeed one can show that for any $\kappa_1,\kappa_2>0$ such that 
	\begin{equation*}
	\begin{aligned}
	\kappa_1c_1+\kappa_2 c_2 &= \lambda_1\left( \mathbb{E}\left[\varphi_{X_1}(\kappa_1 A_{11})\varphi_{X_1}(\kappa_2(1-A_{11})) \right]-1\right) \\ &\quad + \lambda_2\left( \mathbb{E}\left[ \varphi_{X_2}(\kappa_1 A_{12})\varphi_{X_2}(\kappa_2(1-A_{12}))\right] -1\right)
	\end{aligned}
	\end{equation*} it holds that
	\begin{equation*}
	\Psi_{\wedge,\text{sim}}(u) \leq e^{-(\kappa_1 b_1 + \kappa_2 b_2) u},\quad u>0.
	\end{equation*}
	As this has no implications for the considered queueing model we will not go into further details here.
\end{remark}

To derive the asymptotics of $\Psi_{\wedge}$, $\Psi_{\wedge,\text{sim}}$ and $\Psi_{\vee}$ we rely on results from \cite{Avram2009}, which lead to the following Theorem.

\begin{theorem}\label{Theorem_Light_Asymptotics} 
		Assume the claim size variables $X_1,X_2$ fulfill\eqref{eq-lightcondition} and assume there exist (unique) solutions $\kappa_1,\kappa_2>0$ to \eqref{eq_ThLightCond}. Then \begin{align*}
		\Psi_\vee(u) &\sim C_1\cdot e^{-\kappa_1 b_1 u } + C_2\cdot e^{-\kappa_2 b_2 u}, \\ 
		\Psi_\wedge(u) & = o\left(C_1 \cdot e^{-\kappa_1 b_1 u} + C_2 \cdot e^{-\kappa_2b_2 u}\right), \\
\text{and}\quad 	\Psi_{\wedge,\text{sim}}(u) & = o\left(C_1 \cdot e^{-\kappa_1 b_1 u} + C_2 \cdot e^{-\kappa_2b_2 u}\right),
		\end{align*} with $C_1,C_2$ given in \eqref{eq_LightConstants}. 
\end{theorem}

\subsection{Results in the switch context}\label{S4b}

Again, using Lemma \ref{Corollary_Duality} we summarize our findings from the last subsection to obtain the following corollary on the asymptotic behavior of the workload barrier exceedance probabilities in the switching model defined in Section \ref{S2a}.

\begin{corollary}[Asymptotics and bounds of the exceedance probabilities for light-tailed jobs] \label{Cor-lightasympswitch}
	Assume the workload variables $X_1,X_2$ are light-tailed such that \eqref{eq-lightcondition} holds and assume there exist (unique) solutions $\kappa_1,\kappa_2>0$ to \eqref{eq_ThLightCond}.
	Then the workload exceedance probabilities  \eqref{eq-singleexceed}, \eqref{eq-orexceed}, and \eqref{eq-andexceed} fulfil
	\begin{align*}
	\Upsilon_i(b_i u) &\leq e^{-\kappa_i b_i u} \quad\text{for all }u_i>0, i=1,2,\quad \\
\text{and} \quad 	\Upsilon_\vee(u)&\leq (e^{-\kappa_1 b_1 u} + e^{-\kappa_2 b_2 u})\wedge 1 \quad \text{for all }u>0. 
		\end{align*}
		Further, with $C_i, i=1,2,$ as in \eqref{eq_LightConstants},
	it holds
	\begin{align}
	\Upsilon_i(b_i u) &\sim C_i e^{- \kappa_i b_i u}, \quad i=1,2, \nonumber \\
	\Upsilon_\vee(u) &\sim C_1\cdot e^{-\kappa_1 b_1 u } + C_2\cdot e^{-\kappa_2 b_2 u}, \label{eq-lightorasymp} \end{align}
	while the probability that both workloads exceed their barrier fulfills
	\begin{align*}
\Upsilon_\wedge(u) & = o\left(C_1 \cdot e^{-\kappa_1 b_1 u} + C_2 \cdot e^{-\kappa_2b_2 u}\right).
	\end{align*}
\end{corollary}

\begin{remark}
	Note that the light-tail assumption \eqref{eq-lightcondition} does not necessarily imply existence of $\kappa_1,\kappa_2>0$ solving \eqref{eq_ThLightCond}. Assuming for $j=1,2$ the slightly stronger condition
	\begin{center}\begin{minipage}{15cm}
			Either $$\varphi_{X_j}(x_j)<\infty\quad \text{ for all }x_j<\infty,$$ or there exists $x_j^\ast<\infty$ such that $$\varphi_{X_j}(x_j)<\infty\; \text{ for all } x_j<x_j^\ast \quad \text{and} \quad \varphi_{X_j}(x_j)=\infty\; \text{ for all } x_j\geq x_j^\ast.$$	
	\end{minipage}\end{center}
	however is sufficient for  existence of $\kappa_1,\kappa_2>0$. \\
	In case that the above condition fails, i.e. for some $j\in\{1,2\}$ there exists $x_j^\ast$ such that $\varphi_{X_j}(x_j)<\infty$ for all  $x_j\leq x_j^\ast$ and $\varphi_{X_j}(x_j)=\infty$ for all $x_j> x_j^\ast$, then existence of $\kappa_1,\kappa_2$ depends on the chosen parameters of the model; see e.g. \cite[Chapter IV.6a]{asmussenalbrecher} for a more thorough discussion of this.
\end{remark}

\begin{remark}
If $\kappa_1 b_1 \neq \kappa_2b_2$ then the summand of lower order on the right hand side of \eqref{eq-lightorasymp} can be omitted in the asymptotic equivalence. Thus, in contrast to the regularly varying case, the vector $\mathbf{b}$ here is crucial for the exact asymptotic behavior and contributes more than just inside the constant.  \\
On the other hand we immediately see that, given two job distributions and hence given $\kappa_1,\kappa_2>0$, we can choose $b_1,b_2$ in order to minimize the joint exceedance probabilities. The optimal $\mathbf{b}$ then solves  
\begin{equation*} b_1\kappa_1 = b_2\kappa_2, \quad \text{i.e.} \quad b_1= \frac{\kappa_2}{\kappa_1+\kappa_2},\quad \text{and} \quad b_2= \frac{\kappa_1}{\kappa_1+\kappa_2}, \end{equation*}
which leads to
\begin{align*}
\Upsilon_\vee(u) \sim & (C_1+C_2) e^{-\frac{\kappa_1\kappa_2}{\kappa_1+\kappa_2}\cdot u}, \\ 
\text{while} \quad \Upsilon_\wedge(u) = & \mathit{o}\left( e^{-\frac{\kappa_1\kappa_2}{\kappa_1+\kappa_2}\cdot u} \right). 
\end{align*}
\end{remark}

	\section{Examples and simulation study}\label{S5}
	\setcounter{equation}{0}
	
	In this section we consider two special choices of the random switch for which we will evaluate the above results and compare to simulated data. The first part is dedicated to the special case of the Bernoulli switch, where the queueing processes become independent of each other. In the second part we discuss the special case of a non-random switch, where every job is shared between the servers with some predefined deterministic proportions. We finish in Section \ref{S5c} with a short comparison to study the influence of the chosen type of randomness on the exceedance probabilities. 
	
	\subsection{The Bernoulli switch}

	The Bernoulli switch does not split any jobs, but assigns the arriving jobs randomly to one of the two servers. More precisely we set
	$$A_{11}=1-A_{21}\sim\text{Bernoulli}(p), \quad \text{and} \quad A_{12}=1-A_{22}\sim\text{Bernoulli}(q),$$
	independent of each other with $p,q\in[0,1]$. This yields independence of the components of the process $(\bR(t))_{t\geq 0}$ 
	 which can now be represented as
	\begin{align*}
	R_1(t)&=\sum_{k=1}^{N_{1}^{(1)}(t)} X'_{1,k} + \sum_{\ell=1}^{N_{2}^{(1)}(t)} X'_{2,\ell}-tc_1 \quad \text{and} \quad R_2(t)=\sum_{k=1}^{N_{1}^{(2)}(t)} X''_{1,k} + \sum_{\ell=1}^{N_{2}^{(2)}(t)} X''_{2,\ell}-tc_2, 
	\end{align*}
	where $X'_{j,k}$ and $X''_{j,k}$ are independent copies of $X_{j,k}$, $j=1,2$, $k\in\NN$, and the counting processes $(N_1^{(1)}(t))_{t\geq 0}$, $(N_1^{(2)}(t))_{t\geq 0}$, $(N_2^{(1)}(t))_{t\geq 0}$, and $(N_2^{(2)}(t))_{t\geq 0}$ are independent Poisson processes with rates $\frac{\lambda_1 p}{\lambda_1+\lambda_2}$,  $\frac{\lambda_1 (1-p)}{\lambda_1+\lambda_2}$, $\frac{\lambda_2 q}{\lambda_1+\lambda_2}$, and $\frac{\lambda_2 (1-q)}{\lambda_1+\lambda_2}$, respectively. 
	In particular from \eqref{eq-andexceed} and \eqref{eq-includeexclude1} we obtain in the Bernoulli switch
	\begin{equation}\label{eq-includeexcludeBernoulli} \Upsilon_{\wedge}(u)=\Upsilon_1(b_1 u) \Upsilon_2(b_2 u) = \Upsilon_1(b_1 u) + \Upsilon_2(b_2 u)- \Upsilon_\vee(u),\end{equation}
	and hence $\Upsilon_{\wedge}(u)$ and $\Upsilon_{\vee}(u)$ can be expressed in terms of $\Upsilon_1(b_1 u)$, and $\Upsilon_2(b_2 u)$.  Thus, although the Bernoulli switch is not covered by Theorem \ref{Theorem_Subexponential_Risk_Or} in the subexponential setting, as \eqref{eq_theta_result} is not fulfilled, via \eqref{eq-includeexcludeBernoulli} one can still calculate the asymptotics using Lemma \ref{lem-singleruinasymptotic_subexp}.
		
	\begin{figure}[!thb]
		\centering
		\includegraphics[scale=0.5]{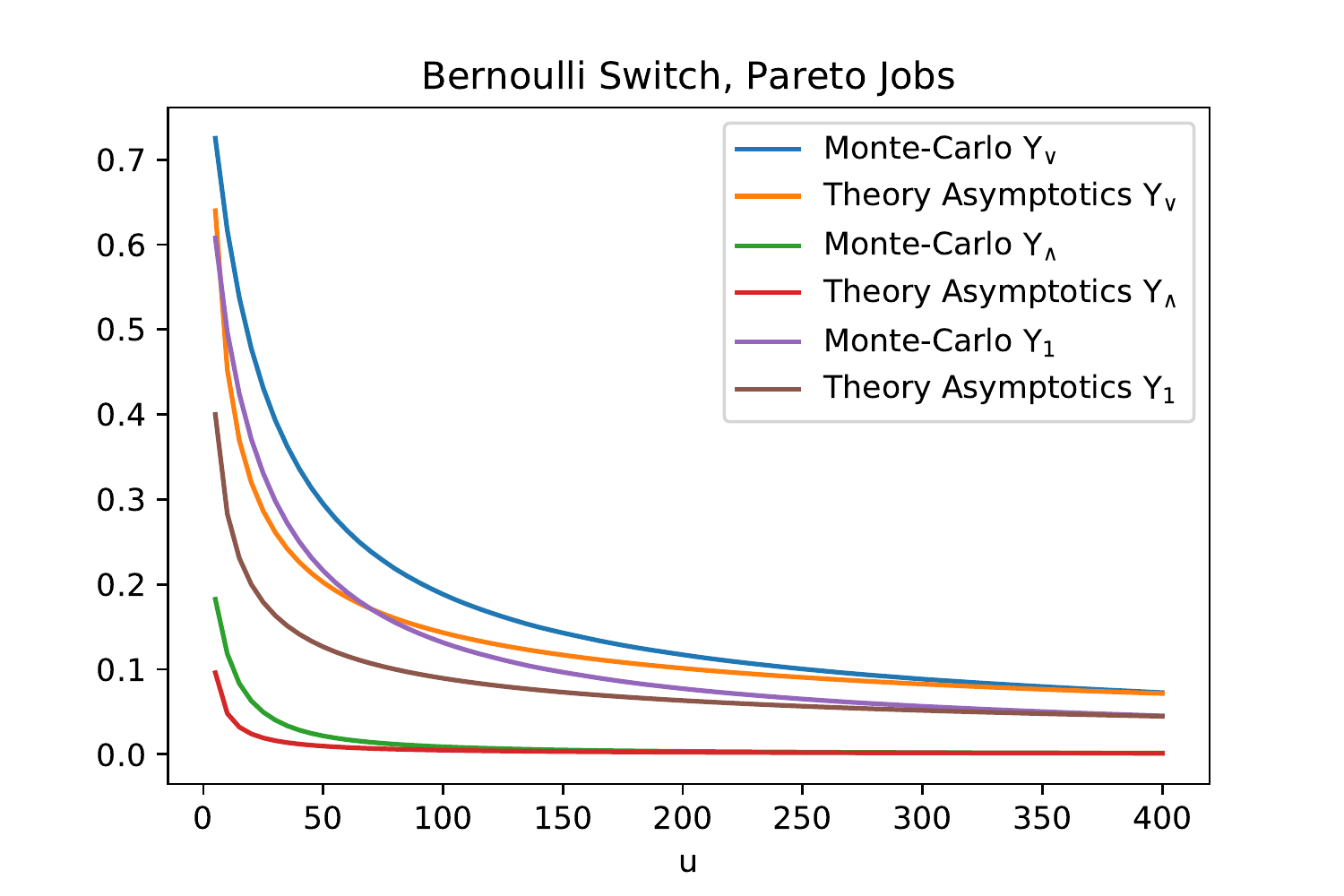}
		\includegraphics[scale=0.5]{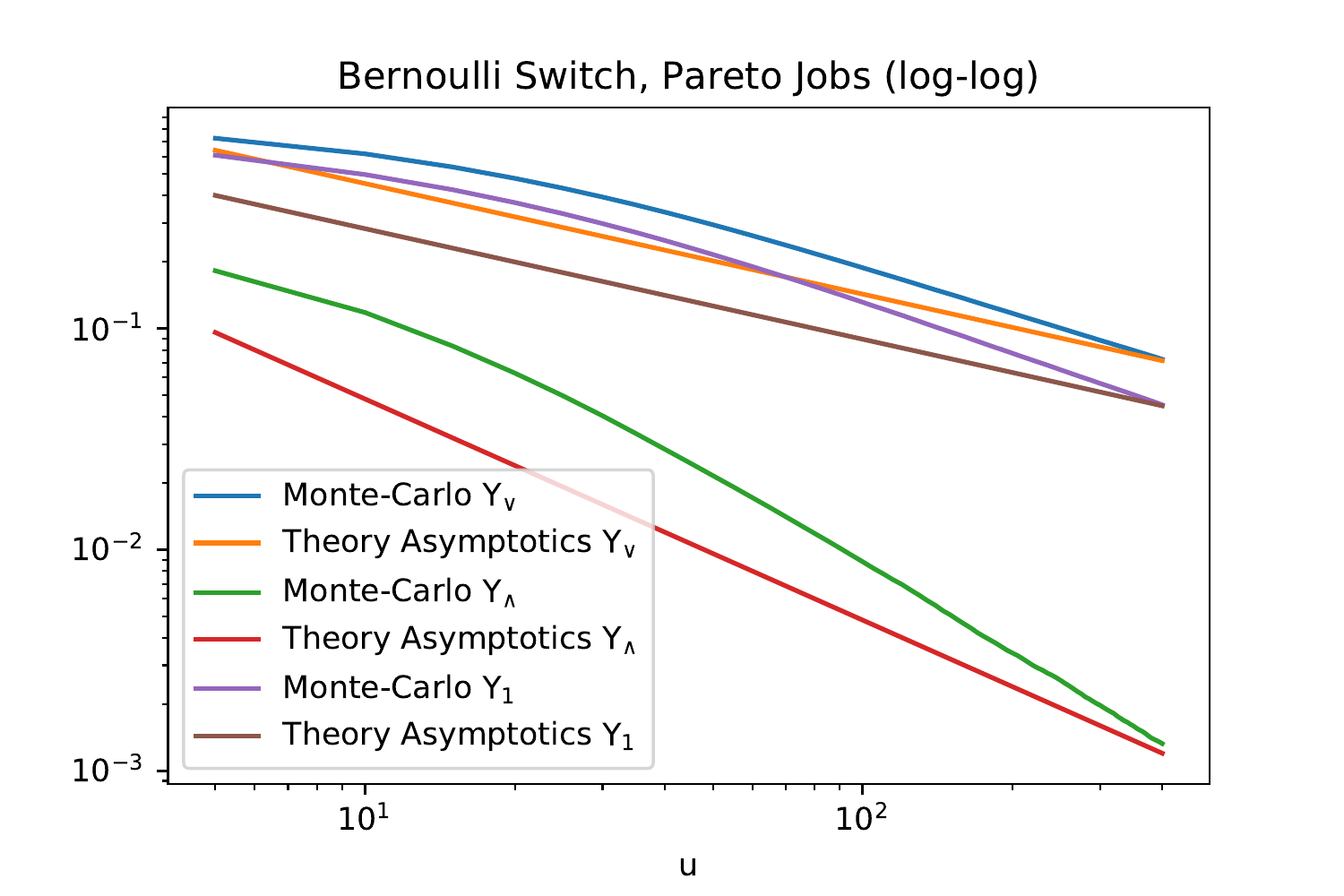}
		\caption{ Simulated exceedance probabilities in the Bernoulli switch in comparison to the obtained asymptotics in natural scaling (left) and as log-log plot (right). \\
			Here job sizes are Pareto distributed with $\overline{F}_1(x)=x^{-3/2}$, $x\geq 1$, and $\overline{F}_2(x)=4x^{-2}$, $x\geq 2$. Further $\lambda_1=\lambda_2=1$, $c_1=5$, $c_2=8$, and $b_1=0.8=1-b_2$. The Bernoulli switch is characterized by $p=0.4$ and $q=0.7$. For these parameters from \eqref{eq-asympBernheavy}  we derive
			$\Upsilon_1(u_1)\sim 0.8\cdot  u_1^{-0.5}$ and $\Upsilon_2(u_2) \sim  0.24  \cdot u_2^{-0.5}$ such that $\Upsilon_\wedge (u) \sim 0.48 \cdot u^{-1}$ and $\Upsilon_\vee(u)\sim 1.431 \cdot u^{-0.5}$ via \eqref{eq-includeexcludeBernoulli}.\\ Note that a direct evaluation of the asymptotics of $\Upsilon_\vee$ as given in Corollary \ref{Theorem_RV_Asymptotics} yields the same result.
		}
		\label{fig_BernPar}
	\end{figure}
	
	\begin{figure}[!thb]
		\centering
		\includegraphics[scale=0.5]{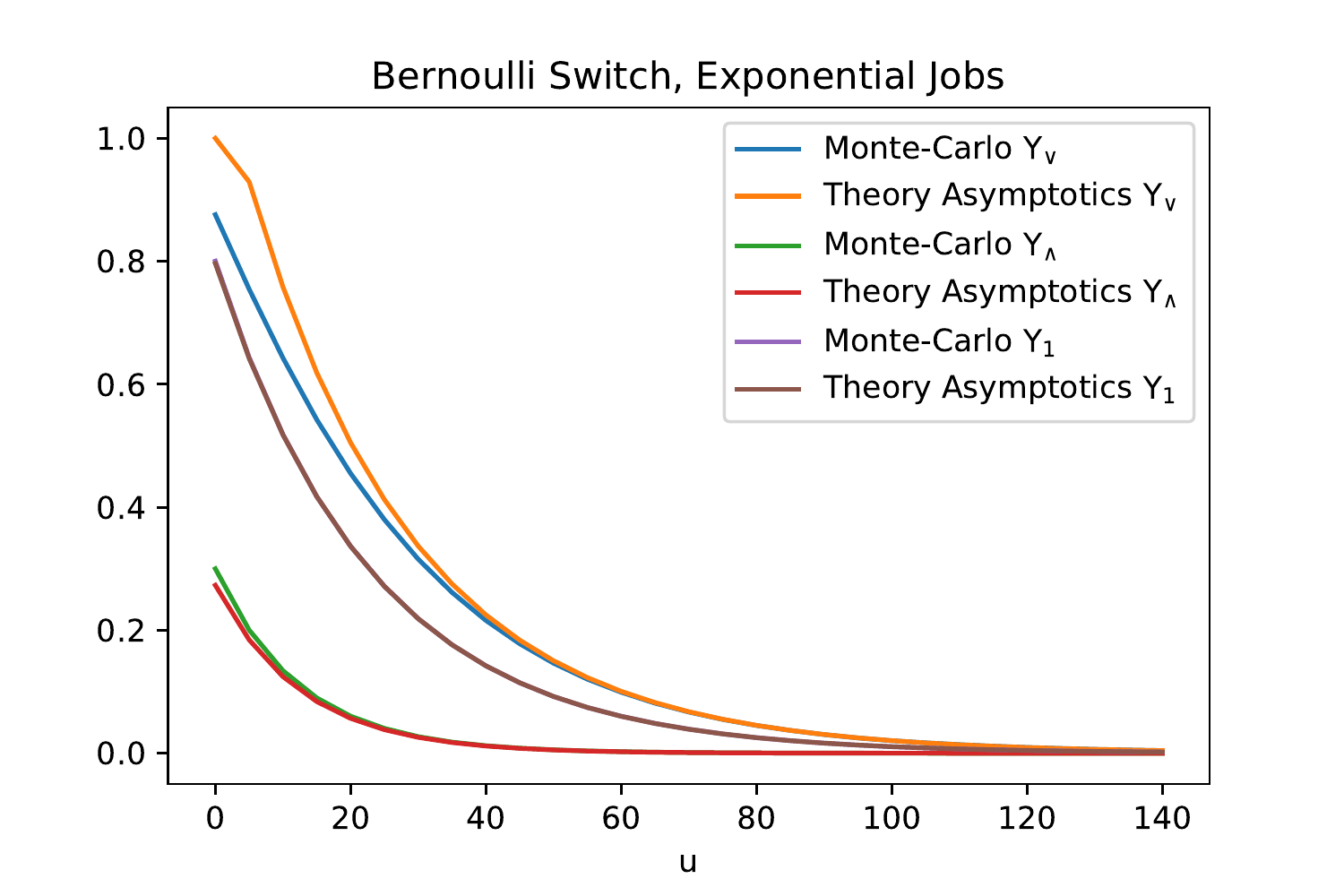}
		\includegraphics[scale=0.5]{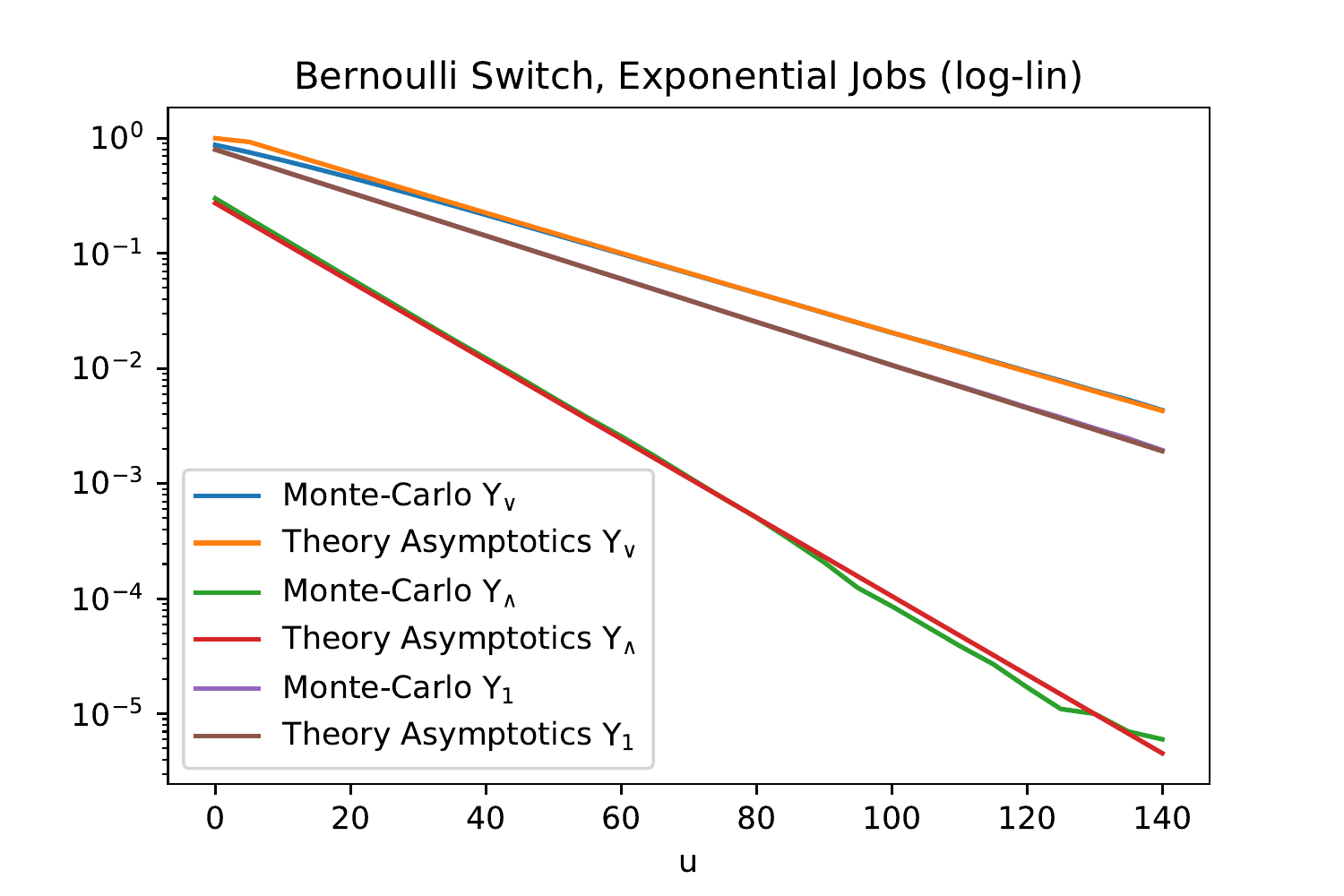}
		\caption{ Simulated exceedance probabilities in the Bernoulli switch in comparison to the obtained asymptotics in natural scaling (left) and as log-linear plot (right). \\
			Here jobs are exponentially distributed with $\overline{F}_1(x)=e^{-x/3}$, $x\geq 0$, and $\overline{F}_2(x)=e^{-x/4}$, $x\geq 0$. Further $\lambda_1=\lambda_2=1$, $c_1=5$, $c_2=8$, and $b_1=0.8=1-b_2$. The Bernoulli switch is characterized by $p=0.4$ and $q=0.7$. For these parameters from \eqref{eq-Bernkappas} we derive $\kappa_1 \approx 0.054$ and $\kappa_2 \approx 0.178$ and \eqref{eq-asympBernlight} yields $\Upsilon_1(u_1)\sim 0.796 \cdot  \exp(-0.054 u_1)$ and $\Upsilon_2(u_2) \sim  0.343 \cdot \exp(-0.178 u_2)$ from which $\Upsilon_\wedge(u)\sim 0.273 \cdot \exp(-0.079u)$ and $\Upsilon_\vee(u)\sim 0.796 \cdot  \exp(-0.043 u) +  0.343 \cdot \exp(-0.036 u)$ via \eqref{eq-includeexcludeBernoulli}. Note that in the latter case we keep both summands, since the exponents are close together.
		}
		\label{fig_BernExp}
	\end{figure}	

Indeed we obtain by direct application of Corollary \ref{Corollary_Subexp_Asymptotics} that if $F_I^1, F_I^2\in \cS$ (which holds in particular if $X_1\in \RV(\alpha_1)$, $X_2\in\RV(\alpha_2)$, $\alpha_1,\alpha_2>1$)
	\begin{equation}\label{eq-asympBernheavy}
	\begin{aligned}
	\Upsilon_1(b_1u)& \sim \frac{\lambda_1 p \int_{b_1 u}^\infty \overline{F}_1(y) \diff y + \lambda_2 q \int_{b_1 u}^\infty \overline{F}_2(y) \diff y}{c_1-\lambda _1 p \EE[X_1] - \lambda_2 q \EE[X_2]}, \\
	\text{and} \quad \Upsilon_2(b_2 u)& \sim \frac{\lambda_1 (1-p) \int_{b_2 u}^\infty \overline{F}_1(y) \diff y + \lambda_2 (1-q) \int_{b_2 u}^\infty \overline{F}_2(y) \diff y}{c_1-\lambda_1 (1-p) \EE[X_1] - \lambda_2 (1-q) \EE[X_2]}. \end{aligned}\end{equation}
	In the light-tailed case an application of Corollary \ref{Cor-lightasympswitch} yields
	\begin{equation}\label{eq-asympBernlight}
	\begin{aligned}
	\Upsilon_1(b_1u)& \sim \frac{c_1 - \lambda_1 p \EE[X_1] - \lambda_2 q \EE[X_2]}{\lambda_1 p\varphi_{ X_1}'(\kappa_1) + \lambda_2q\varphi_{X_2}'(\kappa_1)}e^{-\kappa_1 b_1 u}\\
	\text{and} \quad \Upsilon_2(b_2u)& \sim \frac{c_2 - \lambda_1 (1-p) \EE[X_1] - \lambda_2 (1-q) \EE[X_2]}{\lambda_1 (1-p)\varphi_{ X_1}'(\kappa_2)+ \lambda_2(1-q)\varphi_{X_2}'(\kappa_2)} e^{-\kappa_2 b_2 u},
	\end{aligned}\end{equation}
	as long as there exist $\kappa_1,\kappa_2>0$ such that \eqref{eq_ThLightCond} holds, which in the Bernoulli switch simplifies to
	\begin{equation}\label{eq-Bernkappas}
	\begin{aligned}
	c_1 \kappa_1&= \lambda_1 p (\varphi_{X_1}(\kappa_1)-1) + \lambda_2 q (\varphi_{X_2}(\kappa_1)-1)\\
	\text{and} \quad 	c_2 \kappa_2&= \lambda_1 (1-p) (\varphi_{X_1}(\kappa_2)-1) + \lambda_2 (1-q) (\varphi_{X_2}(\kappa_2)-1).
	\end{aligned}\end{equation}
	The asymptotic behavior of $\Upsilon_\vee$ and $\Upsilon_{\wedge}$ can now be described via \eqref{eq-includeexcludeBernoulli}.

	In Figures \ref{fig_BernPar} and \ref{fig_BernExp} we compare the asymptotics in the Bernoulli switch obtained in this way with data that has been simulated using standard Monte-Carlo techniques. As one can see in all cases the obtained asymptotics fit the data very well for $u$ large enough.

\subsection{The deterministic switch}\label{S5b}

The deterministic switch is characterized by setting
\begin{align*}
	A_{11}&=d_1=1-A_{21},\\
	\text{and} \quad A_{12}&= d_2 = 1-A_{22},
\end{align*}
for some predefined constants $d_1,d_2\in[0,1]$.\\
Note that for $\lambda_2=0$, meaning that we have only one source of claims, the corresponding dual risk model coincides with the \emph{degenerate model} considered in \cite{Avram2009, Avram2007, Fossetal}. Allowing two sources of claims, but setting $d_1\in(0,1)$, $d_2=1$ reduces our model to the setting treated in \cite{Badescu2011}. \\
Clearly, for any choice of $d_1,d_2$ in the deterministic switch one can easily evaluate the asymptotics of the exceedance probabilities as given in Corollaries \ref{Corollary_Subexp_Asymptotics} , \ref{Theorem_RV_Asymptotics}, and \ref{Cor-lightasympswitch} since all appearing expectations disappear.

\begin{figure}[!htb]
	\centering
	\includegraphics[scale=0.5]{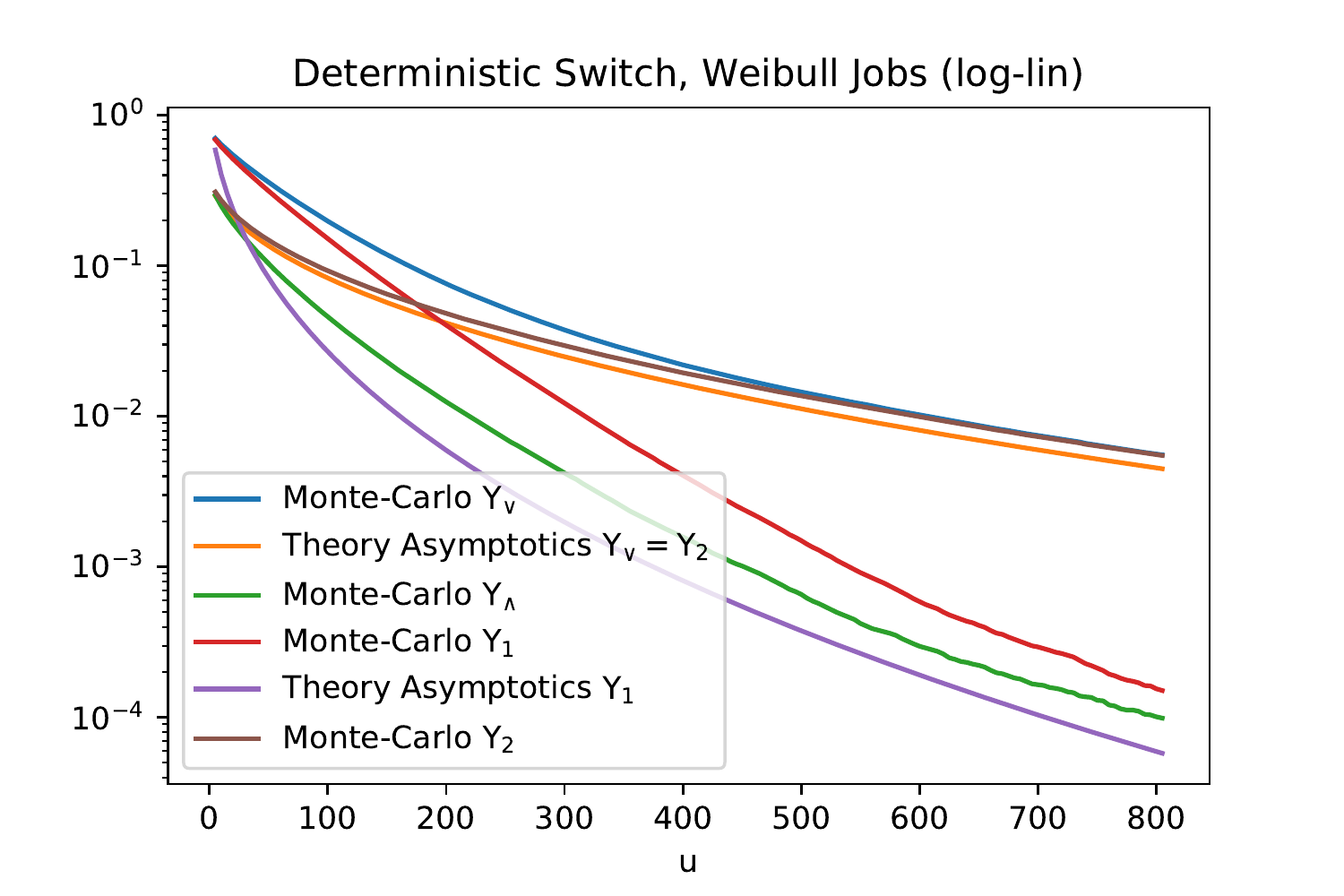}
	\includegraphics[scale=0.5]{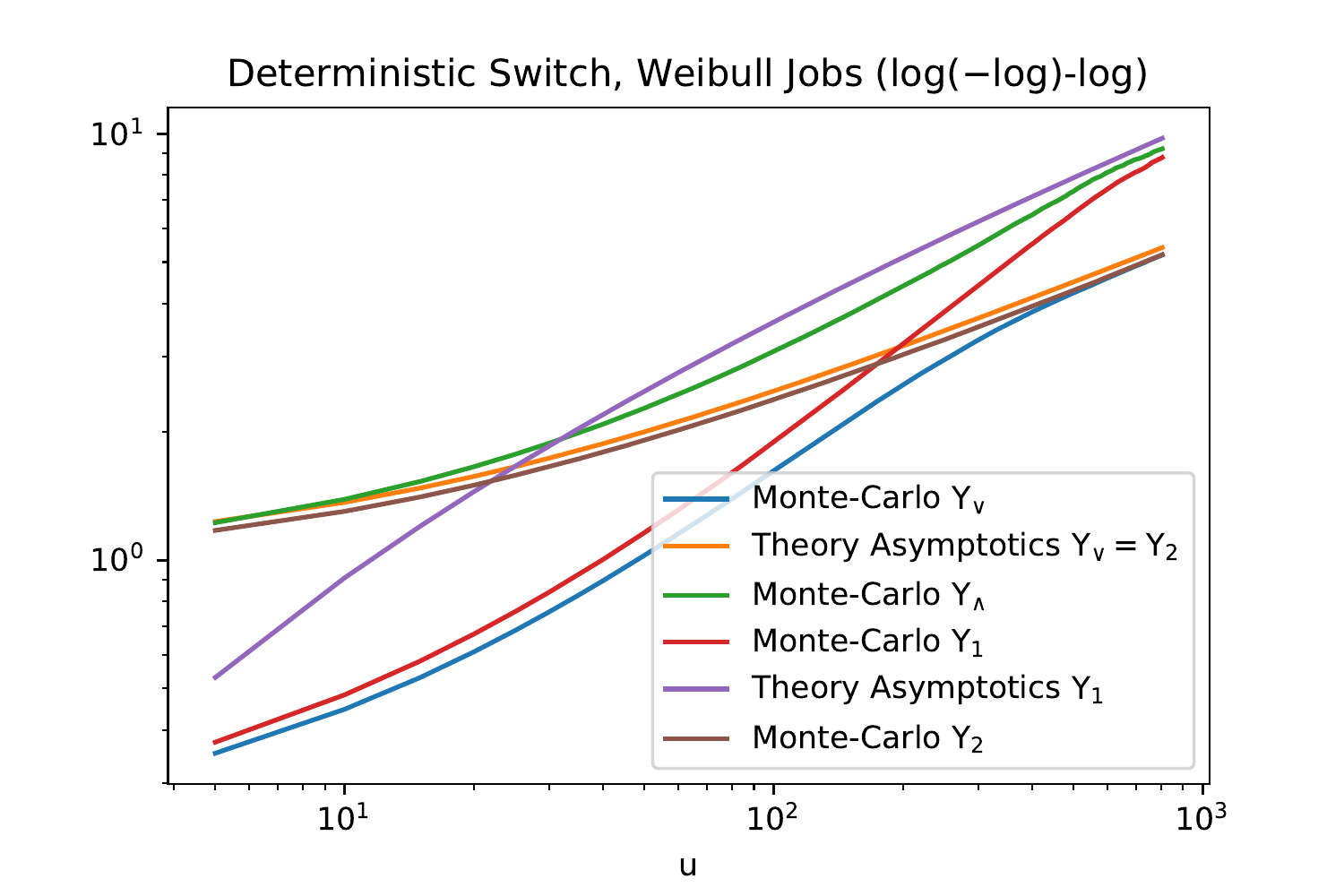}
	\caption{
			Simulated exceedance probabilities in the deterministic switch in comparison to the obtained asymptotics as log-lin plot (left) and log($-$log)-log plot (right), meaning the $y$-axis is scaled as $\log(-\log(y))$, while the $x$-axis is scaled as $\log(x)$. This particular scaling of the axes is chosen in order to get straight lines of ascent $\gamma^{-1}$, for functions of the type $e^{-\sqrt[\gamma]{x}}$. \\ 
			Job sizes are Weibull distributed with $\overline{F}_1(x)=1-\exp((2x)^\frac{1}{3})$, $x\geq 0$, and $\overline{F}_2(x)=1- \exp((x/2)^{\frac{1}{2}})$, $x\geq 0$. Further $\lambda_1=\lambda_2=1$, $c_1=5$, $c_2=8$, and $b_1=0.8=1-b_2$. The deterministic switch is characterized by $d_1=0.4$ and $d_2=0.7$. For these parameters from Corollary \ref{Corollary_Subexp_Asymptotics} we derive $\Upsilon_2(b_2u) \sim \Upsilon_\vee(u) \sim (0.1374\cdot u^{\frac{2}{3}} + 0.31449\cdot u^{\frac{1}{3}} +0.36)\cdot \exp(-0.87358\cdot  u^{\frac{1}{3}})$, and $\Upsilon_1(b_1u) \sim (1.51191 \cdot u^{\frac{2}{3}} + 1.90488\cdot u^{\frac{1}{3}} +  1.2)\cdot \exp( - 1.58740 \cdot u^{\frac{1}{3}})$, while for $\Upsilon_{\wedge}$ no asymptotics are given as \eqref{eq_notAsymptoticEquivalentSW} fails.
	}
	\label{fig_FixWeibull}
\end{figure}

\begin{figure}[!htb]
	\centering
	\includegraphics[scale=0.5]{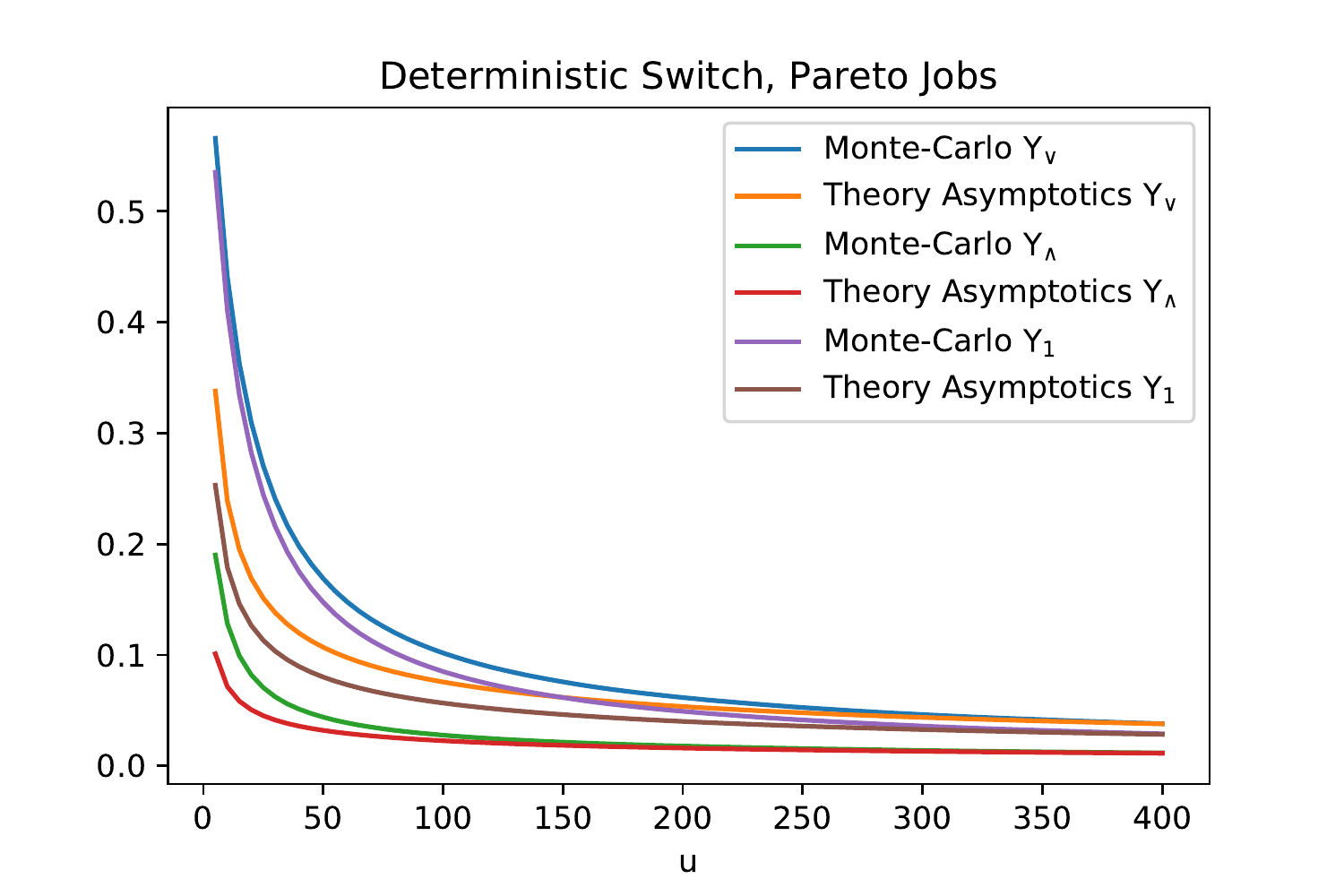}
	\includegraphics[scale=0.5]{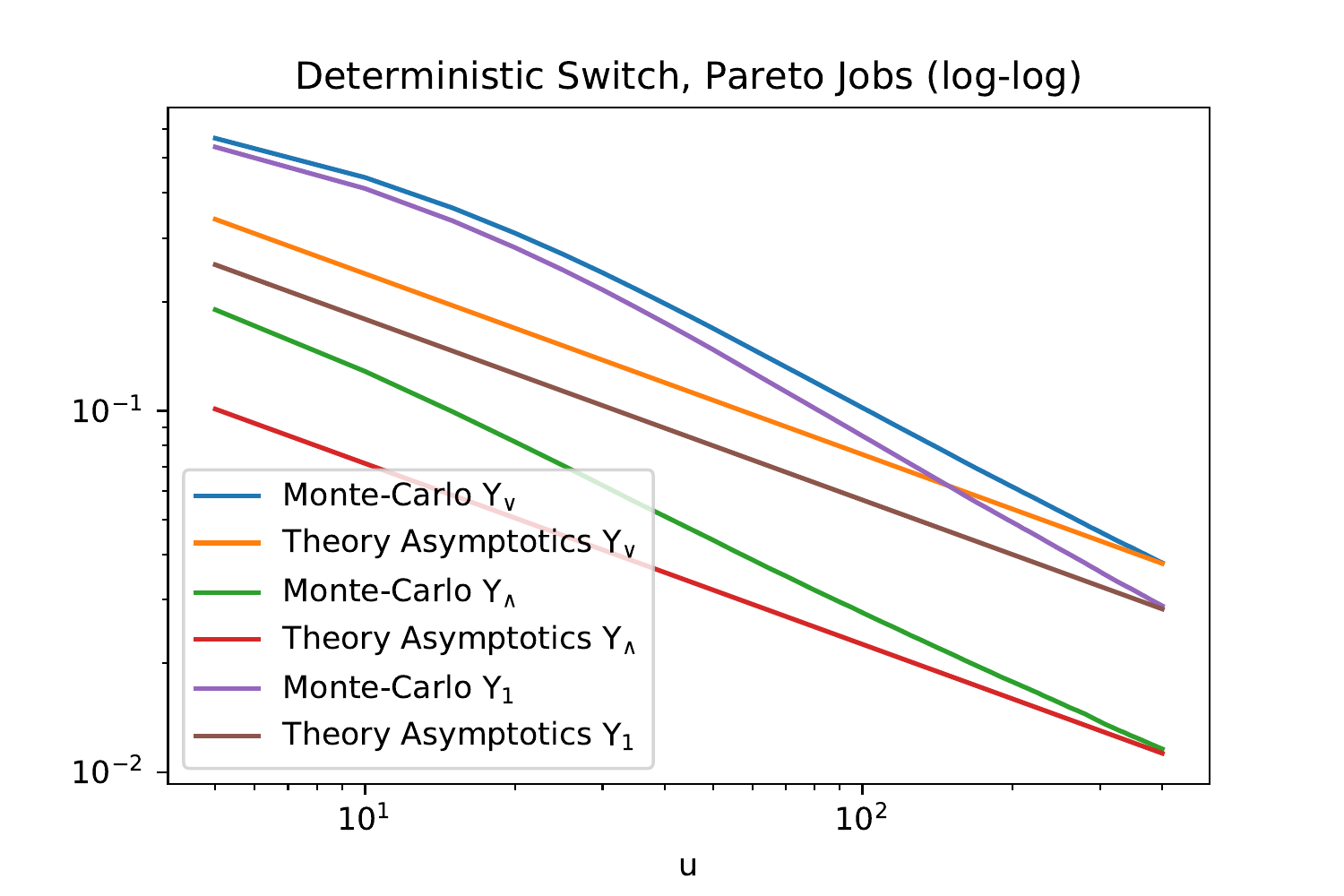}
	\caption{Simulated exceedance probabilities in the deterministic switch in comparison to the obtained asymptotics in natural scaling (left) and as log-log plot (right). \\
		Here - as in Figure \ref{fig_BernPar} - job sizes are Pareto distributed with $\overline{F}_1(x)=x^{-3/2}$, $x\geq 1$, and $\overline{F}_2(x)=4x^{-2}$, $x\geq 2$. Further $\lambda_1=\lambda_2=1$, $c_1=5$, $c_2=8$, and $b_1=0.8=1-b_2$. The deterministic switch is characterized by $d_1=0.4$ and $d_2=0.7$. For these parameters from Corollary \ref{Theorem_RV_Asymptotics} we derive $\Upsilon_1(u_1) \sim 0.506 \cdot u_1^{-0.5}$, and $\Upsilon_2(u_2) \sim 0.186 \cdot u_2^{-0.5} $, while $\Upsilon_\vee(u) \sim  0.756\cdot  u^{-0.5} $, and $\Upsilon_{\wedge} (u)  \sim 0.226\cdot u^{-0.5}$.}
	\label{fig_FixPar}
\end{figure}

\begin{figure}[!htb]
	\centering
	\includegraphics[scale=0.5]{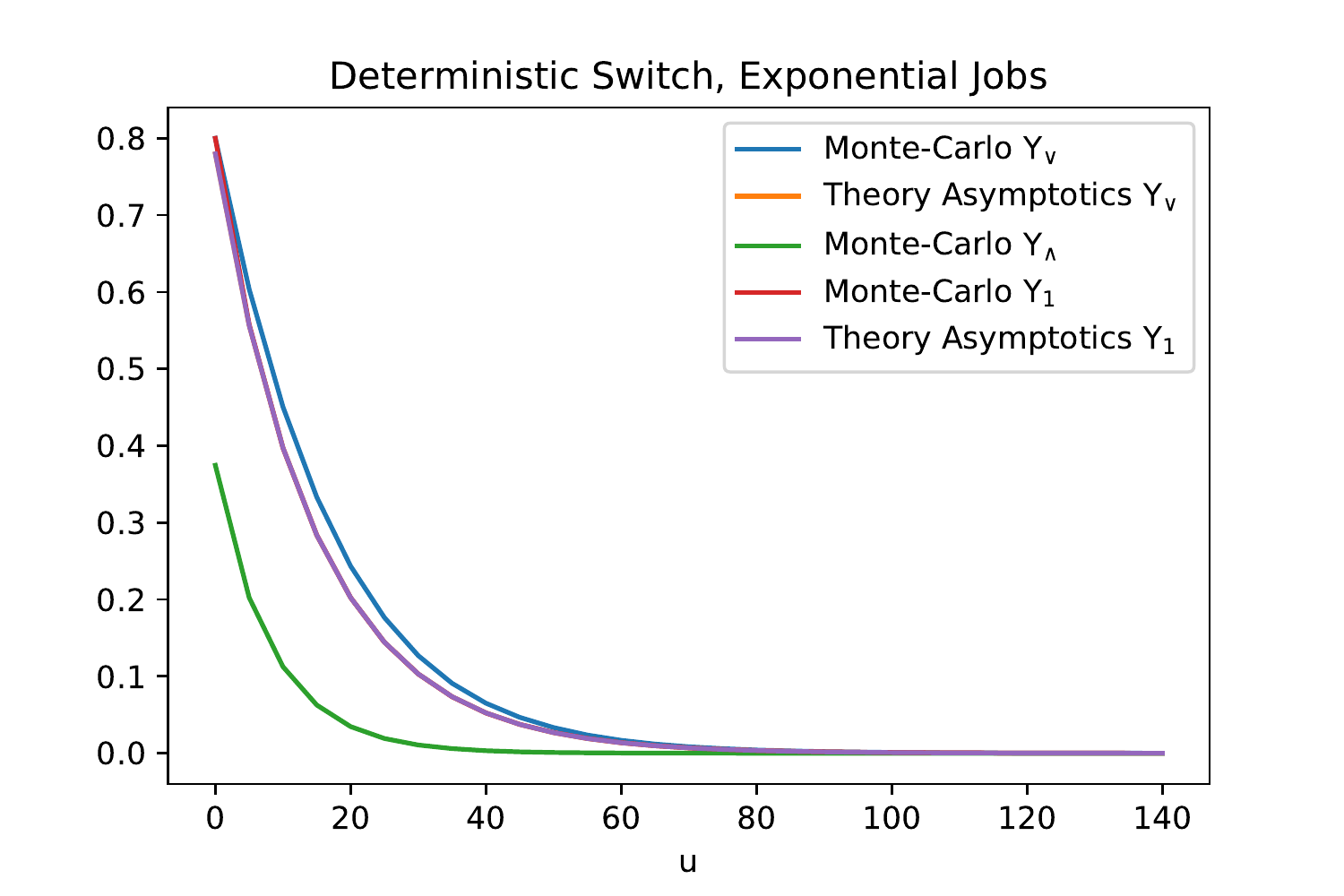}
	\includegraphics[scale=0.5]{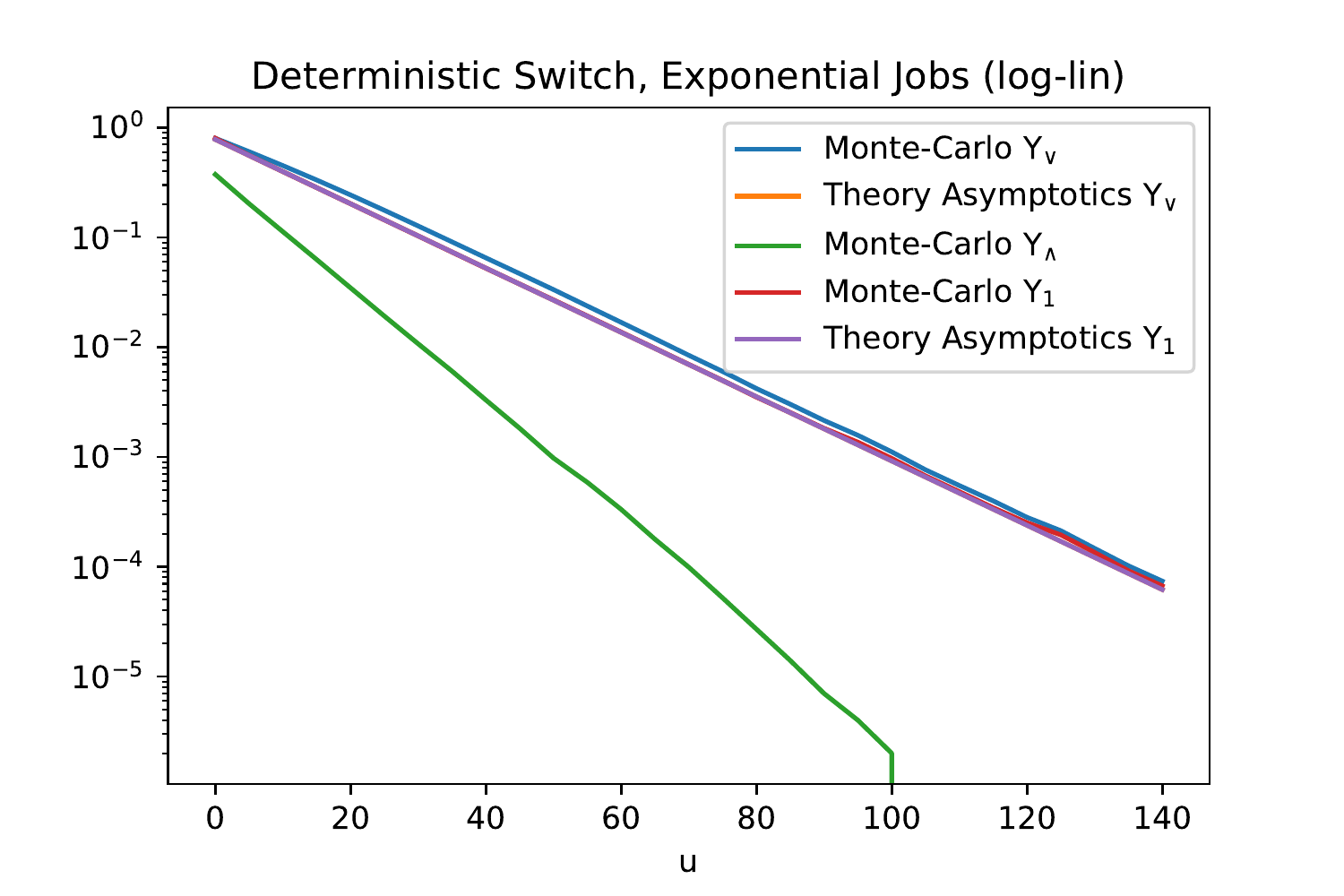}
	\caption{ Simulated exceedance probabilities in the deterministic switch in comparison to the obtained asymptotics in natural scaling (left) and as log-linear plot (right). \\
		Here jobs are - as in Figure \ref{fig_BernExp} - exponentially distributed with $\overline{F}_1(x)=e^{-x/3}$, $x\geq 0$, and $\overline{F}_2(x)=e^{-x/4}$, $x\geq 0$. Further $\lambda_1=\lambda_2=1$, $c_1=5$, $c_2=8$, and $b_1=0.8=1-b_2$. The deterministic switch is characterized by $d_1=0.4$ and $d_2=0.7$. For these parameters from \eqref{eq_ThLightCond} we obtain $\kappa_1\approx 0.084$ and $\kappa_2 \approx 0.383$, which yield  $\Upsilon_1(u_1)\sim 0.78 \cdot  \exp(-0.084 u_1)$ and $\Upsilon_2(u_2) \sim  0.341 \cdot \exp(-0.383 u_2)$, while  $\Upsilon_\vee(u)\sim 0.78 \cdot  \exp(-0.067 u) +  0.341 \cdot \exp(-0.077 u)$ and $\Upsilon_\wedge(u)= o(\exp(-0.067 u))$ by Corollary \ref{Theorem_RV_Asymptotics}.}
	\label{fig_FixExp}
\end{figure}

In Figures \ref{fig_FixWeibull}, \ref{fig_FixPar}, and \ref{fig_FixExp} we compare the asymptotics and bounds in the deterministic switch obtained in this way with data that has been simulated using standard Monte-Carlo techniques. Again simulations and theoretical asymptotics fit well in all cases. 
Note that in contrast to the cases with lighter tails, in the purely subexponential case depicted in Figure \ref{fig_FixWeibull} we observe that $\Upsilon_1$ is close to $\Upsilon_\vee$ for small $u$, but close to $\Upsilon_{\wedge}$ for large $u$. This can be interpreted as follows. While for small $u$ the net working speed $c_i^*$ determines the exceedance probabilities, for large $u$ this becomes less relevant and the workload exceedance is mainly influenced by the heavyness of the tails.
\begin{example} \label{Example_Foss_et_al}
Consider a deterministic switch with $d_1\in (0,1)$ and $\lambda_2=0$, i.e. there is only one source of jobs and the jobs are distributed deterministically to the two servers with proportions $d_1$ and $1-d_1$, and assume that $\tfrac{b_1}{d_1} < \tfrac{b_2}{1-d_1}$, such that the resulting dual risk model properly rescaled coincides with the degenerate model studied in \cite{Avram2009, Avram2007, Fossetal}. Then, applying Theorem \ref{Theorem_Subexponential_Risk_Or} in this setting reproduces the tail behavior of the probability of ruin of at least one component stated in \cite[Cor. 2.2]{Fossetal}. Interestingly, also the ruin probability of both insurance companies one derives in this case from Proposition \ref{Proposition_Subexponential_Risk_And}  coincides with the asymptotics provided in \cite[Eq. (2.9)]{Fossetal}, although the latter corresponds to simultaneous ruin. This suggests that in this special setting the ruin probability of both components and the simultaneous ruin probability of both components are asymptotically equivalent.
\end{example}

\subsection{A comparison of different switches}\label{S5c}

In this section we aim to compare the two above special cases of the Bernoulli switch and the deterministic switch with a non-trivial random switch, which we chose to be a Beta switch characterized by setting 
\begin{align*}
A_{11}&=1-A_{21} \sim \operatorname{Beta}(\beta_1, \gamma_1 ),\\
\text{and} \quad A_{12}&= 1-A_{22}\sim \operatorname{Beta}(\beta_2,\gamma_2)
\end{align*} for some constants $\beta_1,\beta_2,\gamma_1,\gamma_2>0$, where $\text{Beta}(\beta,\gamma)$ is the Beta distribution with density $\frac{\Gamma(\beta+\gamma)}{\Gamma(\beta)\Gamma(\gamma)} x^{\beta-1}(1-x)^{\gamma-1}$, $x\in[0,1]$. \\ 
To keep all examples comparable, we fix $\lambda_1$, $\lambda_2$, $\EE[X_1]$, $\EE[X_2]$, $\mathbb{E}[A_{11}]$ and $\mathbb{E}[A_{12}]$ such that the scenarios only differ in the behavior of the switch and the job sizes. Figure \ref{fig_Comparison} shows the approximate exceedance probabilities obtained by Monte Carlo simulation for the Bernoulli switch, the deterministic switch und two different Beta switches. 

\begin{figure}[!htb]
	\centering
	\includegraphics[scale=0.5]{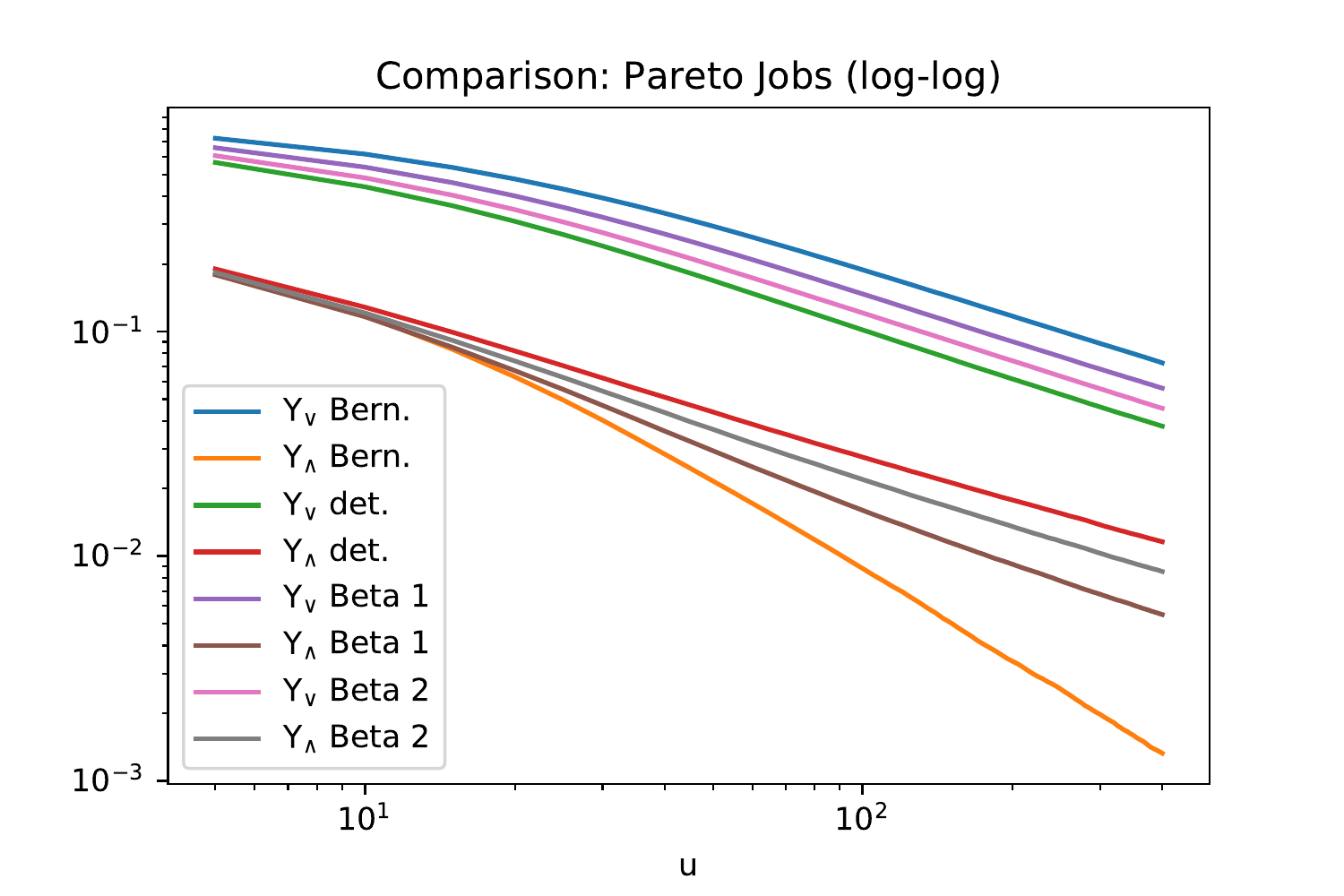}
	\includegraphics[scale=0.5]{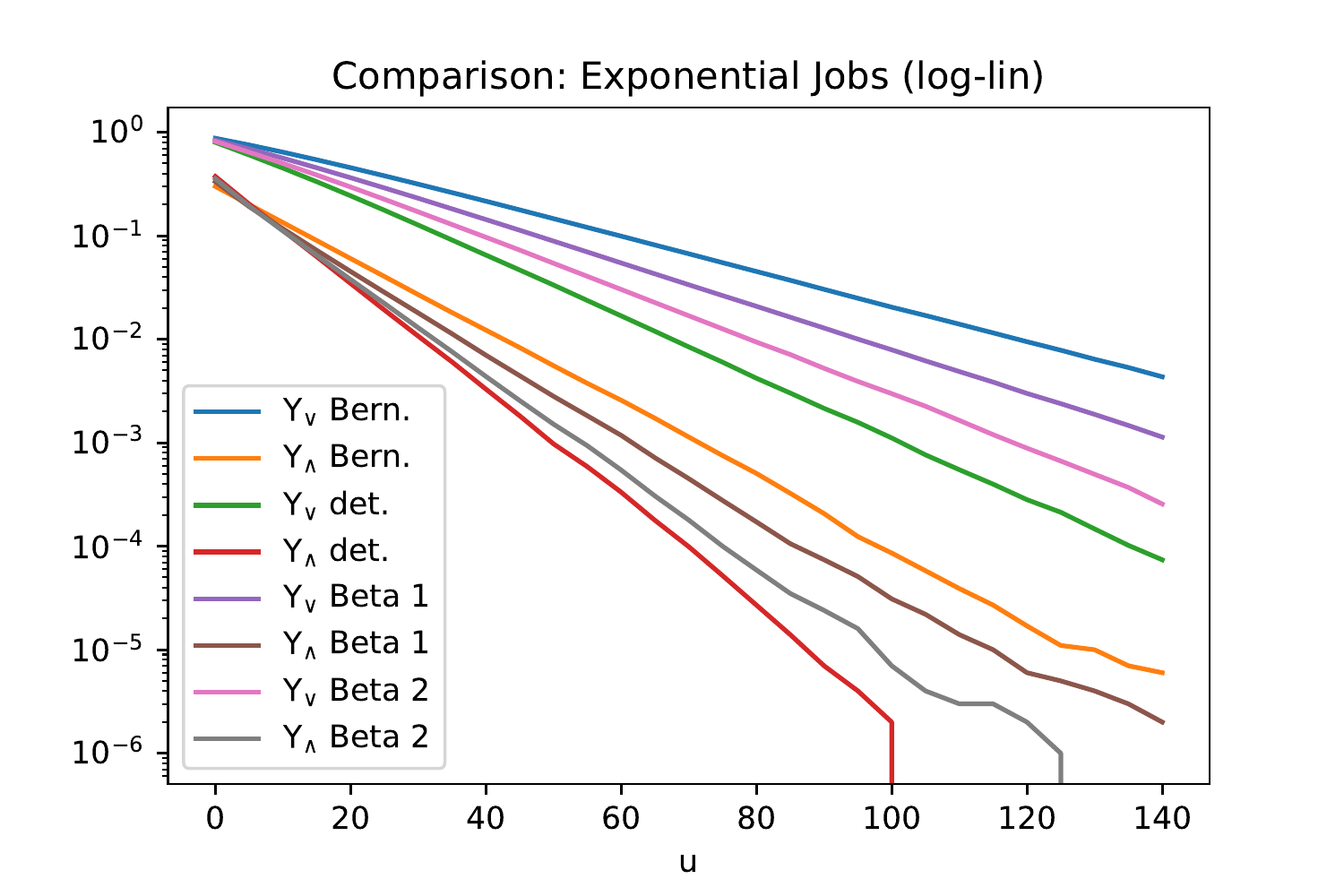}
	\caption{Simulated exceedance probabilities for different switches with heavy-tailed (left, log-log plot) and light-tailed (right, log-linear plot) job sizes. \\
	Throughout $\lambda_1=\lambda_2=1$, $c_1=5$, $c_2=8$, and $b_1=0.8=1-b_2$. On the left  - as in Figures \ref{fig_BernPar} and \ref{fig_FixPar} - job sizes are Pareto distributed with $\overline{F}_1(x)=x^{-3/2}$, $x\geq 1$, and $\overline{F}_2(x)=4x^{-2}$, $x\geq 2$. On the right jobs are - as in Figures \ref{fig_BernExp} and \ref{fig_FixExp} - exponentially distributed with $\overline{F}_1(x)=e^{-x/3}$, $x\geq 0$, and $\overline{F}_2(x)=e^{-x/4}$, $x\geq 0$. The  Bernoulli switch is characterized by $p=0.4$ and $q=0.7$, the deterministic switch is characterized by $d_1=0.4$ and $d_2=0.7$, the Beta switch 1 is characterized by $A_{11}=1-A_{21}\sim \operatorname{Beta}(0.4,0.6)$ and $A_{12}=1-A_{22} \sim \operatorname{Beta}(0.7,0.3)$, and the Beta switch 2 is characterized by $A_{11}=1-A_{21}\sim \operatorname{Beta}(1.5,2.25)$ and $A_{12}=1-A_{22} \sim \operatorname{Beta}(3,9/7)$.}
\label{fig_Comparison}
\end{figure}

As we can see, in the regularly varying case the probability that at least one of the workloads exceeds the barrier $\Upsilon_\vee$ tends to zero with the same index of regular variation for all choices of the random switch. In case of the probability that both components exceed their barrier $\Upsilon_\wedge$, the Bernoulli switch yields a faster decay due to the independence of the two workload processes in this model. \\
Further, the figure indicates the intuitive behavior: The more correlated the co-ordinates of the workload process are, the closer together are $\Upsilon_\vee$ and $\Upsilon_\wedge$. This leads to a trade-off between the two probabilities: Changing the switch towards reducing one probability raises the other and the Beta switches may serve here as a compromise to control both probabilities.

In the light-tailed case the trade-off between $\Upsilon_\vee$ and $\Upsilon_\wedge$ can not be observed. Quite the contrary, the more correlated the co-ordinates of the workload process are, the lower tend to be the exceedance probabilities. Hence in this case the Bernoulli switch yields the highest exceedance probabilities, while the deterministic switch obtains the best results.

Thus, for keeping $\Upsilon_\vee$ small, in general the simple deterministic switch yields good results. On the contrary, if one is interested to keep $\Upsilon_\wedge$ small, the tail-behavior of the appearing jobs is crucial for the choice of the optimal switch. Here again Beta switches or other non-trivial random switches may serve as a compromise in situations where the tail-behavior of the appearing jobs is unknown.

	\section{Proofs}\label{S6}
	\setcounter{equation}{0}
	
	\subsection{Proofs for Section \ref{S3SE}}\label{S3cSE}	

To prove the asymptotic result for the ruin probability $\Psi_\vee$ as given in Theorem \ref{Theorem_Subexponential_Risk_Or} we need some preliminary definitions and results. 

Let \begin{equation*}
		\mathcal{R} := \{M \subset \mathbb{R}^2 : ~ M\text{ open, increasing, } M^c \text{ convex, } \mathbf{0}\notin\overline{M}\}
	\end{equation*} be a family of open sets, where \emph{increasing} means, that for each $\mathbf{x}\in M$ and $\mathbf{d}\geq \mathbf{0}$ we have $\mathbf{x}+\mathbf{d}\in M$. 
	Let $\Pi(\diff x)$ be a probability measure on $\mathbb{R}^2_{\geq 0}$. For $M \in\mathcal{R}$ we define a cdf on $[0,\infty)$ by setting
	\begin{equation*}
		F_M(t) := 1-\Pi(tM), ~ t\geq 0. 
	\end{equation*} Then, following \cite[Def. 4.6]{Samorodnitsky2016}, if $F_M \in\mathcal{S}$ we say that $\Pi \in\mathcal{S}_M$. Furthermore $\Pi$ is \emph{multivariate subexponential}, if $\Pi\in\mathcal{S}_\mathcal{R} := \bigcap_{M\in\mathcal{R}} \mathcal{S}_M$. \\ \smallskip 
Throughout this section we consider the specific sets 
\begin{equation} \label{eq_def_ruinset_A}
L:=\mathbb{R}^2\backslash\mathbb{R}^2_{\geq \mathbf{0}} \qquad \text{and} \qquad M := \mathbf{b}- L. 
\end{equation}
Then clearly $M\in \mathcal{R}$ and by \cite[Rem. 4.1]{Samorodnitsky2016} also $uM\in\mathcal{R}$ for all $u>0$. \\
Moreover, we specify $\Pi$ to be the probability measure of the claims $\mathbf{A}\mathbf{B}\mathbf{X}$ in the dual risk model described in Section \ref{S2b}. Then we can prove some basic relationships in the upcoming lemma.
\begin{lemma} \label{Lemma_Pi_explicitely}
Consider the probability measure $\Pi$ as just defined, the set $M$ as in \eqref{eq_def_ruinset_A}, and the constant $\theta$ given in \eqref{eq_theta_result}. Then 
\begin{align*}
	\Pi(u M + v \mathbf{c}^\ast) &=   g(u,v),  \end{align*}
with $g(u,v)$ defined in \eqref{eq_defguv},
and 
\begin{align*}\Pi(\RR_{\geq 0}^2 + v \mathbf{c}^\ast)= \frac{\lambda_1}{\lambda}\cdot \mathbb{E}\left[ \overline{F}_1\left(v\cdot \max\left\{ \tfrac{c_1^\ast}{A_{11}}, \tfrac{c_2^\ast}{A_{21}}\right\}\right)\right] + \frac{\lambda_2}{\lambda}\cdot \mathbb{E}\left[\overline{F}_2\left(v\cdot \max\left\{ \tfrac{c_1^\ast}{A_{12}}, \tfrac{c_2^\ast}{A_{22}}\right\}\right)\right], 
\end{align*} for all $u>0, ~ v\geq 0$ and $\mathbf{c}^\ast\in\mathbb{R}^2_{>\mathbf{0}}$. Moreover 
\begin{equation*}
	\theta = \int_0^\infty \Pi\left( \RR_{\geq 0}^2 + v \mathbf{c}^*\right) \diff v <\infty.
\end{equation*}
\end{lemma}
\begin{proof}
	By definition of $M$, $L$ and $\Pi$ and conditioning on $\mathbf{B}$
	\begin{align*}
		\Pi(u M+v \mathbf{c}^*) &= \mathbb{P}\left( \mathbf{A}\mathbf{B}\mathbf{X} \in u \mathbf{b} +v \mathbf{c}^* - u L \right)  = \mathbb{P}\left( \mathbf{A}\mathbf{B}\mathbf{X} \in u  \mathbf{b} +v \mathbf{c}^* + \mathbb{R}^2\backslash\mathbb{R}^2_{\leq \mathbf{0}}\right) \\ 
		&= \frac{\lambda_1}{\lambda} \mathbb{P}\left( \begin{pmatrix}A_{11} \\ A_{21} \end{pmatrix}X_1 \in u  \mathbf{b} +v \mathbf{c}^* + \mathbb{R}^2\backslash\mathbb{R}^2_{\leq \mathbf{0}}\right)  \\ & \qquad +  \frac{\lambda_2}{\lambda} \mathbb{P}\left( \begin{pmatrix}A_{12} \\ A_{22} \end{pmatrix}X_2 \in u  \mathbf{b} +v \mathbf{c}^* + \mathbb{R}^2\backslash\mathbb{R}^2_{\leq \mathbf{0}}\right),
	\end{align*}
	where for $j=1,2$
	\begin{align*}
		\mathbb{P}\bigg( \begin{pmatrix}A_{1j} \\ A_{2j} \end{pmatrix}X_j \in u \mathbf{b} +v \mathbf{c}^* + \mathbb{R}^d\backslash\mathbb{R}^d_{\leq \mathbf{0}}\bigg) 
		&= \mathbb{P}\left( ( A_{1j}X_j > ub_1 + vc_1^* ) \text{ or } (A_{2j}X_j>ub_2 + vc_2^*)\right) \\ 
		&=  \mathbb{P}\left( X_j > \min\left\{ \tfrac{ub_1 + vc_1^*}{A_{1j}}, \tfrac{ub_2 + vc_2^*}{A_{2j}}\right\}\right) \\ 
		&= \mathbb{E}\left[\overline{F}_j\left(\min\left\{ \tfrac{ub_1 + vc_1^*}{A_{1j}}, \tfrac{ub_2 + vc_2^*}{A_{2j}}\right\}\right)\right],
	\end{align*}
	which proves the first equation. The second equation follows by an analogue computation. \\
	Lastly, using the obtained expression for $\Pi(\RR_{\geq 0}^2 + v  \mathbf{c}^\ast)$ we may compute using Tonelli's theorem 
	\begin{align*}
		\lefteqn{\int_0^\infty \Pi(\RR_{\geq 0}^2 + v\cdot \mathbf{c}^\ast) dv }\\
		&= \frac{\lambda_1}{\lambda} \mathbb{E}\left[ \int_0^\infty  \overline{F}_1\left(v\cdot \max\left\{ \tfrac{c_1^*}{A_{11}}, \tfrac{c_2^*}{A_{21}}\right\} \right) \diff v \right] +  \frac{\lambda_2}{\lambda} \mathbb{E}\left[\int_0^\infty  \overline{F}_2\left(v\cdot \max\left\{ \tfrac{c_1^*}{A_{12}}, \tfrac{c_2^*}{A_{22}}\right\} \right)  \diff v\right],
	\end{align*}
where for $j=1,2$
\begin{align*}
	\mathbb{E}\left[ \int_0^\infty  \overline{F}_j\left(v\cdot \max\left\{ \tfrac{c_1^*}{A_{1j}}, \tfrac{c_2^*}{A_{2j}}\right\} \right) \diff v \right]  &= \mathbb{E}\left[ \left(\max\left\{ \tfrac{c_1^*}{A_{1j}}, \tfrac{c_2^*}{A_{2j}}\right\}\right)^{-1} \int_0^\infty  \overline{F}_1\left(y \right) \diff y \right]  \\ 
	&= \mathbb{E}\left[\min\left\{ \tfrac{A_{1j}}{c_1^*} , \tfrac{A_{2j}}{c_2^*} \right\}  \right] \mathbb{E}[X_j].  \end{align*}
This proves $\int_0^\infty \Pi(\RR_{\geq 0}^2 + v  \mathbf{c}^\ast) dv = \theta$. Moreover we note that the finite mean of the claim sizes $X_j$ implies finiteness of  $\int_0^\infty \Pi\left(\RR_{\geq 0}^2 + v \mathbf{c}^*\right) \diff v$.
\end{proof}
\begin{proof}[Proof of Theorem \ref{Theorem_Subexponential_Risk_Or}]
Recall the definition of the set $L$ in \eqref{eq_def_ruinset_A} and note that obviously $L$ satisfies \cite[Assumption 5.1]{Samorodnitsky2016}. Furthermore note that 	$u\mathbf{b}-\mathbf{R}(t) \in L$ if and only if $\max_{i=1,2}( R_i(t) -ub_i )>0$, which immediately implies that 
\begin{align*}
	\Psi_\vee(u) &= \mathbb{P}(u\mathbf{b}-\mathbf{R}(t) \in L \text{ for some }t\geq 0) \\
	&= \mathbb{P}(\mathbf{R}(t) \in uM \text{ for some }t\geq 0), \quad u>0.
\end{align*}
Thus, by \cite[Thm. 5.2]{Samorodnitsky2016} we obtain
\begin{equation} \label{eq_subexpasymphelp} \Psi_\vee(u) \sim \int_0^\infty \Pi(u M + v\mathbf{c^*})\diff v, \end{equation}
as soon as we can guarantee that the probability measure on $\RR^2$ defined by
$$\Pi^I(D) =  \theta^{-1} \int_0^\infty \Pi(D+v\mathbf{c}^\ast) \diff v, \quad \text{for any Borel set }D\subset \RR^2_{\geq 0},$$
and $\Pi^I(\RR^2\setminus \RR_{\geq 0}^2)=0$, is in $\cS_M$. This however is by definition equivalent to the assumption that the cdf
$F_M(u)=1-\Pi^I(u M)$, $u\geq 0$, is in $\cS$. Since by Lemma \ref{Lemma_Pi_explicitely} $F_M(u)=F_{\text{subexp}}(u)$ with $F_{\text{subexp}}$ as given in \eqref{eq_def_cdfF_M}, this is assumed in the theorem. 
Lastly, observe that the right hand side of \eqref{eq_subexpasymphelp} equals $\int_0^\infty g(u,v) \diff v$ as shown in Lemma \ref{Lemma_Pi_explicitely}.
\end{proof}
\begin{remark}
Naively one could guess that subexponentiality of the claims $X_1$ and $X_2$ should be enough to obtain subexponentiality of at least $\Pi$. However, as noted in \cite[Remark 4.9]{Samorodnitsky2016}  this is not true in general, because random mixing of two subexponential distributions (as done by our matrix $\mathbf{B}$) leads to a subexponential distribution if and only if the sum of the mixed distributions is subexponential. This again is not true in general.  
\end{remark}
\begin{proof}[Proof of Lemma \ref{lem-singleruinasymptotic_subexp}]
Fix $i\in\{1,2\}$ and assume that $\PP(A_{i1}+A_{i2}=0)< 1.$ Otherwise $R_i(t)$ is monotonely decreasing, $\Psi_i(u)=0$, and the statement is proven. Note that by definition
\begin{align*}
\Psi_i(u)&= \PP\Bigg(\sum_{k=1}^{N(t)} \big((B_{11})_k (A_{i1})_k X_{1,k} + (B_{22})_k (A_{i2})_k X_{2,k}\big) - t c_i >u \text{ for some }t>0\Bigg)\\
&=: \PP\Bigg(\sum_{k=1}^{N(t)} Y_{i,k} - t c_i >u \text{ for some }t>0\Bigg), \quad i=1,2,
\end{align*}
where the random variables $\{Y_{i,k}, k\in\NN \}$ are i.i.d. copies of two generic random variables $Y_i$, $i=1,2$.  
 The corresponding integrated tail function is defined as 
\begin{equation*}
F^{Y_i}_I(x):= \EE[Y_{i}]^{-1} \int_0^x \PP(Y_{i}>y) \diff y, \quad x\geq 0.
\end{equation*}
From \cite[Thm. X.2.1]{asmussenalbrecher} we obtain that if $F^{Y_i}_I \in \mathcal{S}$
\begin{equation} \label{eq-SEsingleruin}
\lim_{u\to\infty} \frac{\Psi_i(u)}{\overline{F^{Y_i}}_I(u)} = \frac{\lambda \EE[Y_{i}] }{c_i-\lambda \EE[Y_{i}]}, \end{equation}
where 
\begin{equation}\label{eq-ewertY}
\EE[Y_i] = \frac{\lambda_1}{\lambda} \EE[A_{i1}]\EE[X_1] + \frac{\lambda_2}{\lambda}  \EE[A_{i2}]\EE[X_2],\end{equation}
such that
$$\frac{\lambda \EE[Y_{i}] }{c_i-\lambda \EE[Y_{i}]}=\frac{\lambda_1 \EE[A_{i1}]\EE[X_1] + \lambda_2 \EE[A_{i2}]\EE[X_2]}{c_i - \lambda_1 \EE[A_{i1}]\EE[X_1] - \lambda_2 \EE[A_{i2}]\EE[X_2]} 
= \frac{\lambda_1 \EE[A_{i1}]\EE[X_1] + \lambda_2 \EE[A_{i2}]\EE[X_2]}{\lambda c_i^\ast}.$$
Further
\begin{align*}
\int_0^x \PP(Y_{i}>y) \diff y&= \frac{\lambda_1}{\lambda} \int_0^x \PP(A_{i1}X_1>y) \diff y + \frac{\lambda_2}{\lambda} \int_0^x \PP(A_{i2}X_2>y) \diff y,
\end{align*}
and since by Tonelli's theorem for all $i,j\in\{1,2\}$
\begin{align*}
	\int_0^x \PP(A_{ij}X_j>y) \diff y &= \int_0^x \EE\left[\overline{F}_j(\tfrac{y}{A_{ij}})\right] \diff y = \EE\left[\int_0^x \overline{F}_j(\tfrac{y}{A_{ij}}) \diff y\right]	\end{align*}
this proves $F_I^{Y_i}= F_I^{i}\in \cS$. Inserting everything in \eqref{eq-SEsingleruin} we obtain
\begin{align*}
	\Psi_i(u) \sim \frac{\lambda_1 \EE[A_{i1}]\EE[X_1] + \lambda_2  \EE[A_{i2}]\EE[X_2] }{c_i- \lambda_1 \EE[A_{i1}]\EE[X_1] + \lambda_2  \EE[A_{i2}]\EE[X_2]} F_I^i(u),
\end{align*} 
which immediately yields the result by \eqref{eq-safetyloading} via substitution with $v = \tfrac{y-b_iu}{c_i^*}$.
\end{proof}
\begin{proof}[Proof of Proposition \ref{Proposition_Subexponential_Risk_And}]
Combining the asymptotics obtained in Theorem \ref{Theorem_Subexponential_Risk_Or}  and Lemma \ref{lem-singleruinasymptotic_subexp} via \eqref{eq-includeexclude1} we obtain due to \eqref{eq_notAsymptoticEquivalentSW}
\begin{align*}
\Upsilon_{\wedge}(u) &\sim  \frac{\lambda_1}{\lambda} \EE\left[ \int_0^\infty \overline{F}_1\left(\tfrac{b_1 u + v c_1^*}{A_{11}} \right) +  \overline{F}_1\left(\tfrac{b_2 u + v c_2^*}{A_{21}} \right) - \overline{F}_1\left(\min\left\{ \tfrac{ub_1 + vc_1^*}{A_{11}}, \tfrac{ub_2 + vc_2^*}{A_{21}}\right\}\right)  \diff  v \right] \\
 &\quad + \frac{\lambda_2}{\lambda} \EE\left[ \int_0^\infty \overline{F}_2\left(\tfrac{b_1 u + v c_1^*}{A_{12}} \right) +  \overline{F}_2\left(\tfrac{b_2 u + v c_2^*}{A_{22}} \right) - \overline{F}_2\left(\min\left\{ \tfrac{ub_1 + vc_1^*}{A_{12}}, \tfrac{ub_2 + vc_2^*}{A_{22}}\right\}\right)  \diff  v \right],
\end{align*}
where 
\begin{align*}
\lefteqn{\overline{F}_i(f_1(v)) + 	\overline{F}_i(f_2(v)) - 	\overline{F}_i(\min\{f_1(v), f_2(v)\})}\\
 &= 	\overline{F}_i(f_1(v)) + 	\overline{F}_i(f_2(v)) - \max\{\overline{F}_i(f_1(v)), \overline{F}_i(f_2(v)) \}	\\
 &= \min \{\overline{F}_i(f_1(v)), \overline{F}_i(f_2(v)) \}\\
 & = \overline{F}_i(\max \{f_1(v), f_2(v) \})
\end{align*}
as $\overline{F}_i$ is monotonely decreasing. This implies the given asymptotics for $\Psi_\wedge$.\\
	If \eqref{eq_notAsymptoticEquivalent} fails, then $\Psi_1(b_1u)+\Psi_2(b_2u)\sim \Psi_\vee(u)$. Therefore we immediately obtain from \eqref{eq-includeexclude2} that
\begin{equation*}
	\lim_{u\to\infty} \frac{\Psi_\wedge(u)}{\Psi_\vee(u)} = \lim_{u\to\infty} \frac{\Psi_1(b_1u)+ \Psi_2(b_2u)}{\Psi_\vee(u)} -1 = 0. 
\end{equation*}
Thus Theorem \ref{Theorem_Subexponential_Risk_Or} implies \eqref{eq-andvanishesSE} which finishes the proof.
\end{proof}

	\subsection{Proofs for Section \ref{S3RV}}\label{S3c}
	
	We start to prove the first statement of Theorem \ref{Theorem_Asymptotics} which we restate below as Lemma \ref{Lemma_Univ_RV_implies_Multi_RV}. 
	
	\begin{proposition} \label{Proposition_regular_mixing}
		Let $\mathbf{Z}\in\operatorname{MRV}(\alpha,\mu)$ be a random vector in $\mathbb{R}^{d}$ and let $\bM$ be a random $(q\times d)$-matrix independent of $\mathbf{Z}$. Let \begin{equation*} \tilde{\mu}({}\cdot{}):=\mathbb{E}[\mu\circ \bM^{-1}({}\cdot{})], \end{equation*}
		where $\bM^{-1}(\cdot)$ denotes the preimage under $\bM$. If $\mathbb{E}[\nrm{\bM}^\gamma]<\infty$ for some $\gamma>\alpha$ and $\tilde{\mu}(\mathcal{B}_1^c)>0$ then $\bM\mathbf{Z}\in\operatorname{MRV}(\alpha,\mu^*)$ with \begin{equation*}
		\mu^*({}\cdot{}) := \frac{1}{\tilde{\mu}(\mathcal{B}_1^c)} \cdot \tilde{\mu}({}\cdot{}), \end{equation*}
		where $\mathcal{B}_1^c:=\{x\in\mathbb{R}^q: ~ \nrm{x}>1\}$ denotes the complement of the unit sphere in $\mathbb{R}^q$.
	\end{proposition}
	\begin{proof}
		First note that our definition of regular variation corresponds to Definition 2.16 (Theorem 2.1.4 (i)) in \cite{Basrack2000}, setting $E=\mathcal{B}_1^c$, which implies $\mathbb{P}(\mathbf{Z} \in tE) = \mathbb{P}(\nrm{\mathbf{Z}}>t)$. Now double application of \cite[Proposition 2.1.18]{Basrack2000} implies the statement, since for $M\subseteq\RR^2$ measurable and bounded away from $\mathbf{0}$ 
		\begin{equation*}
		\frac{\mathbb{P}(\mathbf{M}\mathbf{Z}\in tM)}{\PP(\nrm{\mathbf{M}\mathbf{Z}}>t)} = \underbrace{\frac{\mathbb{P}(\mathbf{M}\mathbf{Z}\in tM)}{\mathbb{P}(\nrm{\mathbf{Z}}>t)}}_{\to \tilde{\mu}(M)}\cdot \underbrace{ \frac{\mathbb{P}(\nrm{\mathbf{Z}}>t)}{\mathbb{P}(\mathbf{M}\mathbf{Z}\in t\mathcal{B}_1^c)}}_{\to\tilde{\mu}(\mathcal{B}_1^c)^{-1}}. \qedhere \end{equation*} 
	\end{proof}

	\begin{lemma} \label{Lemma_Univ_RV_implies_Multi_RV}
		Consider the notation of Section \ref{S2}. If $X_1$ and $X_2$ are regularly varying in the univariate sense with indices $\alpha_1,\alpha_2$, then there exists a measure $\mu^*$ as in Proposition \ref{Proposition_regular_mixing} such that $\bA\bB\mathbf{X}\in\operatorname{MRV}(\min\{\alpha_1,\alpha_2\},\mu^*)$.
	\end{lemma}
	\begin{proof}
		Obviously  $\mathbf{X}=(X_1,X_2)\in\operatorname{MRV}(\alpha,\mu)$ for some non-null measure $\mu$ concentrated on the axes, and $\alpha=\min(\alpha_1,\alpha_2)$ since the random variables $X_1,X_2$ are independent and both regularly varying with indices $\alpha_1,\alpha_2$. To prove the Lemma it is thus enough to check the prerequisites of Proposition \ref{Proposition_regular_mixing}. Clearly, using the properties of $\bA$ and $\bB$ we compute   $\EE[\|\bA\bB\|^\gamma]=1<\infty$ for any $\gamma$. Further for $M\subseteq \RR^2$ measurable and bounded away from~$\mathbf{0}$ 
		\begin{align*}
		\tilde{\mu}(M) =&\mathbb{E}[\mu\circ (\bA\bB)^{-1}(M)] 
		= \mathbb{E}\left[\mu\left(\left\{\mathbf{x}\in\mathbb{R}^2: ~ \bA\bB\mathbf{x}\in M\right\}\right)\right] \\ 
		=& \frac{\lambda_1}{\lambda}\cdot\mathbb{E}\left[\mu\left(\left\{\mathbf{x}=(x_1,x_2)\in\mathbb{R}^2: ~ \begin{pmatrix} A_{11} x_1 \\ A_{21}x_1\end{pmatrix}\in M\right\}\right)\right] \\ &+\frac{\lambda_2}{\lambda}\cdot\mathbb{E}\left[\mu\left(\left\{\mathbf{x}=(x_1,x_2)\in\mathbb{R}^2: ~ \begin{pmatrix} A_{12} x_2 \\ A_{22}x_2\end{pmatrix}\in M\right\}\right)\right]. 
		\end{align*}
		Thus for $M=\mathcal{B}_1^c$ and recalling property (ii) of the matrix $\bA$ we obtain
		\begin{align*}
		\lefteqn{\tilde{\mu}(\mathcal{B}_1^c)}\\ =&\frac{\lambda_1}{\lambda} \cdot\mathbb{E}[\mu(\{\mathbf{x}=(x_1,x_2)\in\mathbb{R}^2: ~ |x_1|>1\})] 
		+\frac{\lambda_2}{\lambda} \cdot\mathbb{E}[\mu(\{\mathbf{x}=(x_1,x_2)\in\mathbb{R}^2: ~ |x_2|>1\})] \\
		=& \frac{\lambda_1}{\lambda} \cdot \mu((1,\infty)\times\mathbb{R}) + \frac{\lambda_2}{\lambda}\cdot\mu(\mathbb{R}\times(1,\infty)) 
		>0,
		\end{align*}
		where we have used that, due to positivity of $\bX$, $\mu$ is zero on $\RR^2\backslash \RR_{>0}^2$. This finishes the proof.
	\end{proof}
	
	To prove the remainder of Theorem \ref{Theorem_Asymptotics} we will use a result from \cite{Hult2005}. To do so, first recall the bivariate compound Poisson process $\mathbf{R}$ from our dual risk model from Section \ref{S2b}. Let $(T_k)_{k\in\mathbb{N}}$ be the independent identically $\operatorname{Exp}(\lambda)$-distributed interarrival times of the Poisson process $N(t)$, i.e. 
	\begin{equation*}
	N(t) = \sum_{n=1}^\infty \mathds{1}_{\{\sum_{k=1}^n T_k\leq t\}}.
	\end{equation*}
	We define the random walk 
	\begin{equation}\label{eq-def-I}
	\mathbf{S}_n := \sum_{k=1}^n\left(\bA_k\bB_k\mathbf{X}_k-T_k \mathbf{c}\right) + n\cdot\left(\mathbb{E}[T_1]\mathbf{c}-\mathbb{E}[\bA\bB\mathbf{X}]\right),
	\end{equation}
	and directly observe that $(\mathbf{S}_n)_{n\in\NN}$ is compensated, i.e. for all $n\in\NN$
	\begin{align}\label{eq-I-compensate}
	\mathbb{E}[\mathbf{S}_n]
	&= \sum_{k=1}^n \left(\mathbb{E}[\bA_k\bB_k\mathbf{X}_k]-\mathbb{E}[T_k \mathbf{c}]\right) + n\cdot\mathbb{E}[T_1]\mathbf{c}-n \cdot \mathbb{E}[\bA\bB\mathbf{X}] 
	=\mathbf{0}.
	\end{align}
	
	The following Lemma explains the relationship between the risk process $(\mathbf{R}(t))_{t\geq 0}$ and the random walk $(\mathbf{S}_n)_{n\in\NN}$.

	\begin{lemma} \label{Lemma_Ruin_Random_Walk}
		Let $L\subseteq \RR^2$ be a \emph{ruin set}, i.e. assume that
		\begin{enumerate}[(i)]
			\item $L\backslash \mathbb{R}_{<0}^2=$, i.e., $L\cap \mathbb{R}_{<0}^2=\emptyset$, and
			\item  $uL=L$ for all $u>0$. 
		\end{enumerate}
		Then 
		\begin{align*}
		\Psi_L(u) :=&\mathbb{P}\big(\mathbf{S}_n - n\left(\lambda^{-1}\mathbf{c}-\mathbb{E}[\bA\bB\mathbf{X}]\right) \in u(\mathbf{b}+L) \text{ for some }n\in\mathbb{N} \big) \\ =& \mathbb{P}\big(\mathbf{R}(t)- u\mathbf{b} \in L \text{ for some }t\geq 0 \big).
		\end{align*}
	\end{lemma}
	\begin{proof}
		Recall from \eqref{eq-def-R} that $\mathbf{R}(t) = \sum_{k=1}^{N(t)} \bA_k\bB_k\mathbf{X}_k - t\mathbf{c}$ where $\mathbf{c}=(c_1,c_2)^\top \in\RR_{\geq 0}^2$. Thus by assumption {\it (i)} $\bR(t)$ may enter $L$ only by a jump and since $N(t)\overset{t\nearrow\infty}{\longrightarrow}\infty$ a.s. we get
		\begin{align*}
		\lefteqn{\{\mathbf{R}(t) -u\mathbf{b}\in L\text{ for some }t\geq 0  \}}\\
		&= \left\{\sum_{k=1}^{N(t)}\bA_k\bB_k\mathbf{X}_k-t\mathbf{c} \in u\mathbf{b}+ L\text{ for some }t\geq 0 \right\} \\
		&=\left\{\sum_{k=1}^{n} (\bA_k\bB_k\mathbf{X}_k- T_k\mathbf{c}) \in u(\mathbf{b}+ L)\text{ for some } n\in\mathbb{N} \right\}\\
		&= \left\{\sum_{k=1}^{n} (\bA_k\bB_k\mathbf{X}_k- T_k\mathbf{c}) +(n-n)\left(\lambda^{-1}\mathbf{c}-\mathbb{E}[\bA\bB\mathbf{X}]\right) \in u(\mathbf{b}+L) \text{ for some }n\in\mathbb{N}\right\} \\
		&= \{\mathbf{S}_n - n\left(\lambda^{-1}\mathbf{c}-\mathbb{E}[\bA\bB\mathbf{X}]\right) \in u(\mathbf{b}+L) \text{ for some }n\in\mathbb{N}\},
		\end{align*}
		which yields the claim.
	\end{proof}
	
	We proceed with a Lemma that specifies the ruin sets that we are interested in.
	
	\begin{lemma} \label{Lemma_Ruin_Sets}
		Let \begin{align*}
		L_\vee &:= \{(x_1,x_2)\in\mathbb{R}^2: ~ x_1>0 \vee x_2>0\} = \RR^2\backslash \RR_{\leq 0}^2, \\ 
		\text{and} \quad L_{\wedge,\text{sim}} &:= \{(x_1,x_2)\in\mathbb{R}^2: ~ x_1>0 \wedge x_2>0\}= \RR_{>0}^2,
		\end{align*}
		then
		\begin{align*}
		\Psi_{L_\vee} (u) = \Psi_\vee(u), \quad 
		\text{and} \quad 	\Psi_{L_{\wedge,\text{sim}}} (u) &= \Psi_{\wedge,\text{sim}}(u), \quad u>0.
		\end{align*}
	\end{lemma}
	
	\begin{proof}
		Clearly 
		$$\mathbb{P}\big(\mathbf{R}(t) - u\mathbf{b}\in  L_\vee  \text{ for some }t\geq 0\big)  = \mathbb{P}\left(\max_{i=1,2}(R_i(t)-u_i) >0 \text{ for some }t\geq 0\right)$$
		which is $\Psi_{L_\vee} (u) = \Psi_\vee(u)$.
		The second equality follows analogously.
	\end{proof}

	\begin{proposition} \label{Proposition_limit}
		Let the claim size variables $X_1,X_2$ be regularly varying, i.e. $X_{j} \in \operatorname{RV}(\alpha_j)$ for $\alpha_j>1$. Then $\bA\bB\mathbf{X}\in\operatorname{MRV}(\min(\alpha_1,\alpha_2),\mu^*)$ for a suitable measure $\mu^*$. Further, recall $\mathbf{c}^*=(c_1^*,c_2^*)^\top\in \mathbb{R}^2_{>0}$ from \eqref{eq-safetyloading}. Let $L\subseteq\mathbb{R}^2$ be a ruin set in the sense of  Lemma \ref{Lemma_Ruin_Random_Walk} and assume additionally:
		\begin{itemize}
			\item[(iii)] For all $\mathbf{a}\in\mathbb{R}^2_{>0}$  \begin{equation*}
			\mu^*(\partial(\mathbf{a}+L))=0.
			\end{equation*}
			\item[(iv)] The set $\mathbf{b}+L$ is $\mathbf{p}$-increasing for all $\mathbf{p}\in\mathbb{R}^2_{>0}$, i.e., for all $v\geq 0$ it holds that 
			\begin{equation*} 
			\mathbf{x} \in \mathbf{b}+ L \quad \text{implies} \quad \mathbf{x} +v\mathbf{p} \in \mathbf{b}+ L. \end{equation*}
		\end{itemize} 
		Then
		\begin{align*}
		\lim_{u\to\infty}\frac{\Psi_L(u)}{u\cdot\mathbb{P}(\nrm{\bA\bB\mathbf{X}}>u)}  &= \int_0^\infty \mu^*(v\mathbf{c^*}+\mathbf{b}+L)\diff v.
		\end{align*}
	\end{proposition}

	\begin{proof}
		That $\bA\bB\mathbf{X}\in \operatorname{RV}(\min(\alpha_1,\alpha_2),\mu^*)$ has been shown in Lemma \ref{Lemma_Univ_RV_implies_Multi_RV}. Recalling the definitions of $\mathbf{S}_n$ and $\Psi_L(u)$ we may write
		\begin{align*}
		\Psi_L(u) &= \mathbb{P}\left(\mathbf{S}_n - n\left(\lambda^{-1}\mathbf{c}-\mathbb{E}[\bA\bB\mathbf{X}]\right) \in u(\mathbf{b}+L) \text{ for some }n\in\mathbb{N} \right) \\
		&= \mathbb{P}\left(\sum_{k=1}^n\mathbf{Y}_k - n\mathbf{c}^*\in u(\mathbf{b}+L) \text{ for some }n\in\mathbb{N} \right),
		\end{align*} for i.i.d. random vectors  \begin{align*}
		\mathbf{Y}_k &= \bA_k\bB_k\mathbf{X}_k-T_k \mathbf{c} + \lambda^{-1}\mathbf{c}-\mathbb{E}[\bA\bB\mathbf{X}].
		\end{align*}
		All the other prerequisites ensure that we may apply \cite[Thm. 3.1 and Rem. 3.2]{Hult2005} to obtain the desired asymptotics.
	\end{proof}
	
	The following Lemma justifies the usage of Proposition \ref{Proposition_limit} for our problem.
	
	\begin{lemma}\label{lem-checkconditions}
		The sets $L_\vee$ and $L_{\wedge,\text{sim}}$ from Lemma~\ref{Lemma_Ruin_Sets} satisfy conditions (i)-(iv) of Lemma~\ref{Lemma_Ruin_Random_Walk} and Proposition~\ref{Proposition_limit}.
	\end{lemma}
	\begin{proof}
		Properties {\it (i)}, {\it (ii)} and {\it (iv)} are obvious. Consider {\it (iii)}. Fix an arbitrary $\mathbf{a}=(a_1,a_2)^\top\in\mathbb{R}^2_{>\mathbf{0}}$. It holds that \begin{equation*}
		\partial (\mathbf{a}+L) = \mathbf{a}+\partial(L) \end{equation*} and we have 
		\begin{align*}
		\partial(L_\vee) & = \{\mathbf{x}\in\mathbb{R}^2: ~(x_1=0 \wedge x_2\leq 0) \vee (x_1\leq 0 \wedge x_2=0) \} \\
		\partial(L_{\wedge,\text{sim}}) & =  \{\mathbf{x}\in\mathbb{R}^2: ~(x_1=0 \wedge x_2\geq 0) \vee (x_1\geq 0 \wedge x_2=0) \}.
		\end{align*}
		Set \begin{align*}
		M_{1}(\ba) &:= \{(x_1,x_2)\in\mathbb{R}^2: ~ x_1\leq a_1 ~\wedge ~ x_2=a_2\}, \\
		M_{2}(\ba) &:= \{(x_1,x_2)\in\mathbb{R}^2: ~ x_1= a_1 ~\wedge ~ x_2\leq a_2\},  
		\end{align*}
		such that $\mathbf{a}+\partial L_\vee = M_{1}(\ba)\cup M_{2}(\ba)$. 
		Now consider the set $M_1(\ba)$. Let $t\in(1,\infty)\cap \mathbb{Q}$, then 
		\begin{equation*}
		t M_{1}(\ba) = \{(x_1,x_2)\in\mathbb{R}^2: ~ x_1\leq t a_1 \wedge x_2=ta_2 \}. \end{equation*}
		Thus for $t_1\neq t_2$ we have $t_1M_{1}(\ba) \cap t_2M_{1}(\ba) = \emptyset$. Further the set $\bigcup_{t\in(1,\infty)\cap\mathbb{Q}} tM_{1}(\ba)$ is obviously bounded away from zero, since $(a_1,a_2)>\mathbf{0}$. We thus obtain
		\begin{align*}
		\infty > \mu^\ast \left(\bigcup_{t\in(1,\infty)\cap\mathbb{Q}} tM_{1}(\ba)\right) &= \sum_{t\in(1,\infty)\cap\mathbb{Q}} \mu^\ast (tM_{1}(\ba)) \\
		&= \sum_{t\in(1,\infty)\cap\mathbb{Q}} t^{-\min\{\alpha_1,\alpha_2\}}\mu^\ast(M_{1}(\ba))\\
		&=\mu^\ast (M_{1}(\ba)) \sum_{t\in(1,\infty)\cap\mathbb{Q}} t^{-\min\{\alpha_1,\alpha_2\}}. \end{align*}
		Since the last sum is infinite, $\mu^\ast(M_{1}(\ba))$ must be zero. The same argument applied to $M_2(\ba)$ thus yields the result for $L_\vee$. The proof for $L_{\wedge,\text{sim}}$ is analogue.
	\end{proof}
	
	\begin{proof}[Proof of Theorem \ref{Theorem_Asymptotics}]
		The first statement has been shown in Lemma \ref{Lemma_Univ_RV_implies_Multi_RV}. The asymptotics for $\Psi_\vee$ and $\Psi_{\wedge,\text{sim}}$ are direct consequences of Lemma~\ref{lem-checkconditions} and Proposition~\ref{Proposition_limit}. 
	\end{proof}
	
	For the proof of Proposition \ref{Cor-AsymptoticsRuin} we will use the following lemma.

	\begin{lemma}\label{lem-regvarlimits}
		Let $f,g$ be regularly varying with indices $\alpha,\beta>0$ and set 
		\begin{equation*}
		\zeta:=\lim_{t\to\infty} \frac{\lambda_1f(t)}{\lambda_2g(t)}\in[0,\infty],
		\end{equation*}
		for $\lambda_1,\lambda_2>0$, such that  $\zeta\in(0,\infty)$ clearly implies $\alpha=\beta$. Then for any constants $\gamma_1,\gamma_2>0$ 
		\begin{equation*}
		\lim_{t\to\infty} \frac{\lambda_1f(\gamma_1t)+ \lambda_2g(\gamma_2t)}{\lambda_1f(t)+\lambda_2g(t)} = \frac{\zeta \gamma_1^\alpha +\gamma_2^\beta}{1+\zeta},
		\end{equation*}
		where we interpret $\frac{\infty\cdot x}{\infty}:=x$. 
	\end{lemma}
	\begin{proof}
		Obviously it holds that 
		\begin{equation*}
		\frac{\lambda_1f(\gamma_1t)+ \lambda_2g(\gamma_2t)}{\lambda_1 f(t)+\lambda_2g(t)} = \frac{\frac{f(\gamma_1t)}{f(t)}}{1+\frac{\lambda_2 g(t)}{\lambda_1 f(t)}} + \frac{\frac{g(\gamma_2t)}{g(t)}}{1+\frac{\lambda_1f(t)}{\lambda_2 g(t)}}\underset{t\to\infty}\longrightarrow \frac{\gamma_1^{\alpha}}{1+\zeta^{-1}} + \frac{\gamma_2^\beta}{1+\zeta}=\frac{\zeta \gamma_1^\alpha +\gamma_2^\beta}{1+\zeta}. \qedhere
		\end{equation*}
	\end{proof}
	
	\begin{proof}[Proof of Proposition \ref{Cor-AsymptoticsRuin}]
		We concentrate first on the $\vee$-case and start by determining the constant $C_\vee$. Using the limiting-measure property of $\mu^\ast$,  \eqref{eq-Nenner} and the properties of $\bA$ and $\bB$ we obtain
		\begin{align*}
		\lefteqn{\int_0^\infty  \mu^*(v\mathbf{c}^*+\mathbf{b} +L_{\vee})\diff v } \\ 
		=& \int_0^\infty \lim_{t\to\infty} \frac{\mathbb{P}(\bA\bB\mathbf{X}\in t(v\mathbf{c}^*+\mathbf{b} + L_{\vee}))}{\mathbb{P}(\nrm{\bA\bB\mathbf{X}}>t)} \diff v \\
		=&\int_0^\infty \lim_{t\to\infty} \left( \frac{\frac{\lambda_1}{\lambda}\mathbb{P}\left(\big(\begin{smallmatrix} A_{11} X_1\\ A_{21}X_1\end{smallmatrix}\big)\in t(v\mathbf{c}^*+\mathbf{b} + L_{\vee})\right)}{\frac{\lambda_1}{\lambda}\cdot \mathbb{P}(X_1>t) +\frac{\lambda_2}{\lambda}\cdot \mathbb{P}(X_2>t)} +\frac{\frac{\lambda_2}{\lambda}\mathbb{P}\left(\big(\begin{smallmatrix} A_{12} X_2\\ A_{22}X_2\end{smallmatrix}\big)\in t(v\mathbf{c}^*+\mathbf{b} + L_{\vee})\right)}{\frac{\lambda_1}{\lambda}\cdot \mathbb{P}(X_1>t) +\frac{\lambda_2}{\lambda}\cdot \mathbb{P}(X_2>t)}\right) \diff v. 
		\end{align*} 
		Now recall  that $L_\vee = \{(x_1,x_2)\in\mathbb{R}^2: ~ x_1>0 \vee x_2>0\}$ which yields
		\begin{equation*}
		t(v\mathbf{c}^*+\mathbf{b}+L_\vee) = \left\{(x_1,x_2)\in\mathbb{R}^2: ~(x_1>  tvc_1^\ast+tb_1) \vee (x_2>tvc_2^\ast+tb_2)\right\}. \end{equation*}
		Hence 
		\begin{align*}
		\mathbb{P}\left(\big(\begin{smallmatrix} A_{11}X_1 \\ A_{21}X_1\end{smallmatrix}\big) \in t(v\mathbf{c}^* + \mathbf{b}+L_\vee)\right)  
		&= \mathbb{P}\left(A_{11}X_1>t(vc_1^*+b_1) \vee A_{21}X_1>t(vc_2^*+b_2)\right) \\ 
		&= \mathbb{P}\left(X_1> \min\left\{\tfrac{t(vc_1^*+b_1)}{A_{11}},\tfrac{t(vc_2^*+b_2)}{A_{21}} \right\}\right) \\ 
		&= \mathbb{P}\left(X_1> t\cdot \min\left\{\tfrac{vc_1^*+b_1}{A_{11}},\tfrac{vc_2^*+b_2}{A_{21}} \right\}\right).
		\end{align*}
		A similar computation for $\big(\begin{smallmatrix} A_{12}X_2 \\ A_{22}X_2\end{smallmatrix}\big)$ thus leads to
		\begin{align*}
		\lefteqn{ \mu^*(v\mathbf{c}^*+\mathbf{b} +L_{\vee}) } \\ 
		=& \lim_{t\to\infty} \left(\frac{\lambda_1 \mathbb{P}\left(X_1> t\cdot \min\left\{\frac{vc_1^*+b_1}{A_{11}},\frac{vc_2^*+b_2}{A_{21}} \right\}\right) }{\lambda_1 \mathbb{P}(X_1>t) +\lambda_2 \mathbb{P}(X_2>t)} +\frac{\lambda_2 \mathbb{P}\left(X_2> t\cdot \min\left\{\frac{vc_1^*+b_1}{A_{12}},\frac{vc_2^*+b_2}{A_{22}} \right\}\right) }{\lambda_1 \mathbb{P}(X_1>t) +\lambda_2 \mathbb{P}(X_2>t)}\right)\\
		=& \lim_{t\to\infty} \int_{\mathbf{a}\in\mathbb{A}} \frac{\lambda_1 \mathbb{P}\left(X_1> t\cdot \min\left\{\frac{vc_1^*+b_1}{a_{11}},\frac{vc_2^*+b_2}{a_{21}} \right\}\right) + \lambda_2 \mathbb{P}\left(X_2> t\cdot \min\left\{\frac{vc_1^*+b_1}{a_{12}},\frac{vc_2^*+b_2}{a_{22}} \right\}\right) }{\lambda_1 \mathbb{P}(X_1>t) +\lambda_2 \mathbb{P}(X_2>t)}  \diff \PP_\bA \\
		=&\int_{\mathbf{a}\in\mathbb{A}} \lim_{t\to\infty} \frac{\lambda_1 \mathbb{P}\left(X_1> t\cdot \min\left\{\frac{vc_1^*+b_1}{a_{11}},\frac{vc_2^*+b_2}{a_{21}} \right\}\right) + \lambda_2 \mathbb{P}\left(X_2> t\cdot \min\left\{\frac{vc_1^*+b_1}{a_{12}},\frac{vc_2^*+b_2}{a_{22}} \right\}\right) }{\lambda_1 \mathbb{P}(X_1>t) +\lambda_2 \mathbb{P}(X_2>t)}  \diff \PP_\bA,
		\end{align*}
		where $\mathbb{P}_\bA({}\cdot{})$ denotes the probability measure induced by $A$ and $\mathbb{A}$ denotes the set of all possible realisation of $\bA$. Hereby the second equality has been obtained by conditioning on $\bA=\ba$ while the last equality follows from Lebesgue's theorem of dominated convergence. Note that Lebesgue's theorem is applicable since
		\begin{align*}
		\lefteqn{\frac{\lambda_1 \mathbb{P}\left(X_1>t\min\left\{\frac{vc_1^*+b_1}{\mathbf{a}_{11}},\frac{vc_2^*+b_2}{\mathbf{a}_{21}}\right\}\right)+\lambda_2 \mathbb{P}\left(X_2>t\min\left\{\frac{vc_1^*+b_1}{\mathbf{a}_{12}},\frac{vc_2^*+b_2}{\mathbf{a}_{22}}\right\}\right)}{\lambda_1  \mathbb{P}(X_1>t) +\lambda_2  \mathbb{P}(X_2>t)}}\\ 
		&\leq\frac{\lambda_1 \mathbb{P}\left(X_1>t\min\left\{vc_1^*+b_1,vc_2^*+b_2\right\}\right)+\lambda_2 \mathbb{P}\left(X_2>t\min\left\{vc_1^*+b_1,vc_2^*+b_2\right\}\right)}{\lambda_1 \mathbb{P}(X_1>t) +\lambda_2 \mathbb{P}(X_2>t)}\\  
		&\leq\frac{\lambda_1 \mathbb{P}\left(X_1>t\min\left\{vc_1^*+b_1,vc_2^*+b_2\right\}\right)}{\lambda_1\cdot \mathbb{P}(X_1>t)} +\frac{\lambda_2 \mathbb{P}\left(X_2>t\min\left\{vc_1^*+b_1,vc_2^*+b_2\right\}\right)}{\lambda_2\cdot \mathbb{P}(X_2>t)}\\ 
		&\to  \left(\min\{vc_1^*+b_1,vc_2^*+b_2\}\right)^{-\alpha_1} + \left(\min\{vc_1^*+b_1,vc_2^*+b_2\}\right)^{-\alpha_2} 
		\end{align*} and thus there exists $t_0>0$ independent of the realisation $\mathbf{a}$ such that for all $t>t_0$ the integrand is smaller than \begin{equation*}
		2 \left( \left(\min\{vc_1^*+b_1,vc_2^*+b_2\}\right)^{-\alpha_1} + \left(\min\{vc_1^*+b_1,vc_2^*+b_2\}\right)^{-\alpha_2} \right),
		\end{equation*} which, as a constant (with respect to $\bA$), is clearly $\mathbb{P}_\bA$-integrable. \\
		By Tonelli's theorem we thus obtain
		\begin{align*} 
		\lefteqn{C_\vee = \int_0^\infty  \mu^*(v\mathbf{c}^*+\mathbf{b} +L_{\vee})\diff v } \\
		&=  \mathbb{E}\left[\int_0^\infty  \lim_{t\to\infty} \frac{\lambda_1 \overline{F}_1\left(t\min\left\{\frac{vc_1^*+b_1}{A_{11}},\frac{vc_2^*+b_2}{A_{21}}\right\}\right)+\lambda_2 \overline{F}_2\left(t\min\left\{\frac{vc_1^*+b_1}{A_{12}},\frac{vc_2^*+b_2}{A_{22}}\right\}\right)}{\lambda_1\cdot \overline{F}_1(t) +\lambda_2\cdot \overline{F}_2(t)} \diff v\right]. 
		\end{align*}
		Applying Lemma \ref{lem-regvarlimits} now yields \eqref{eq-orasymptotic}.\\	
		The proof of \eqref{eq-andsimasymptotics} can be carried out in complete analogy. 
	\end{proof}
	\begin{proof}[Proof of Lemma \ref{lem-singleruinasymptotic}]
It is enough to prove that under the present assumptions also the assumption of Lemma \ref{lem-singleruinasymptotic_subexp} is fulfilled. Hence, we need to show that $X_1\in\operatorname{RV}(\alpha_1), ~ X_2\in\operatorname{RV}(\alpha_2)$ for $\alpha_1,\alpha_2>1$ implies that $F_I^i\in \mathcal{S}$. Recall $Y_i$ and $F^{Y_i}_I=F_I^i$ from the proof of Lemma \ref{lem-singleruinasymptotic_subexp} and assume for the moment, that neither $A_{i1}=0$ a.s., nor $A_{i2}=0$ a.s. Then, using Proposition \ref{Proposition_regular_mixing} and the same argumentation as in the proof of Lemma \ref{Lemma_Univ_RV_implies_Multi_RV} we obtain that $Y_{i}\in \RV(\min\{\alpha_1,\alpha_2\})$. 
Thus the corresponding tail functions of the integrated tail functions $\overline{F^{Y_i}}_I$
are regularly varying as well, with index $-\min\{\alpha_1,\alpha_2\}+1$, which implies $F^i_I \in \mathcal{S}$. 
If $A_{i1}=0$ a.s. then $Y_i=A_{i2}\mathds{1}_{B_{22}=1} X_2$ and clearly $Y_i\in \RV (\alpha_2)$ which again implies $F^i_I \in\mathcal{S}$. 
\end{proof}

	\begin{proof}[Proof of Proposition \ref{prop-andasympruin}]
		Assume \eqref{eq_notAsymptoticEquivalent} holds true. From Lemma \ref{lem-singleruinasymptotic} and its proof we obtain directly as $u=u_1+u_2 \to\infty$
		\begin{align*}
		\Psi_1(b_1 u)+\Psi_2(b_2 u) &\sim \frac{1}{\lambda} \left(\lambda_1 \left( \frac{1}{c_1^\ast} \int_{b_1 u}^\infty \PP(A_{11}X_1>y) \diff y + \frac{1}{c_2^\ast} \int_{b_2 u}^\infty \PP(A_{21}X_1>y) \diff y \right)  \right.\\
		&\qquad +\left. \lambda_2 \left(\frac{1}{c_1^\ast} \int_{b_1 u}^\infty \PP(A_{12}X_2>y) \diff y + \frac{1}{c_2^\ast} \int_{b_2 u}^\infty \PP(A_{22}X_2>y) \diff y  \right)\right),
		\end{align*}
		where the first two terms on the right hand side are regularly varying with index $-\alpha_1+1$, while the latter two terms are regularly varying with index $-\alpha_2+1$.\\
		Together with \eqref{eq_Asymptotics_1}, \eqref{eq-Nenner} we thus obtain that as $u\to\infty$
		\begin{align*}
		\Psi_\wedge(u)&= \Psi_1(b_1 u) + \Psi_2(b_2 u) - \Psi_{\vee}(u)\\
		&\sim \frac{1}{\lambda} \left(\lambda_1 \left( \frac{1}{c_1^\ast} \int_{b_1 u}^\infty \PP(A_{11}X_1>y) \diff y + \frac{1}{c_2^\ast} \int_{b_2 u}^\infty \PP(A_{21}X_1>y) \diff y - C_\vee u \overline{F}_1(u) \right)  \right.\\
		&\qquad +\left. \lambda_2 \left(\frac{1}{c_1^\ast} \int_{b_1 u}^\infty \PP(A_{12}X_2>y) \diff y + \frac{1}{c_2^\ast} \int_{b_2 u}^\infty \PP(A_{22}X_2>y) \diff y - C_\vee u \overline{F}_2(u) \right)\right),
		\end{align*}
		where \eqref{eq_notAsymptoticEquivalent} ensures that terms with the same index of regular variation do not cancel out asymptotically.
		Using Tonelli's theorem as in the proof of Lemma \ref{lem-singleruinasymptotic} this yields 
		\begin{align*}
		\Psi_\wedge(u)& \sim \frac{1}{\lambda} \left(\lambda_1 \left( \frac{1}{c_1^\ast} \EE\left[\int_{b_1u}^\infty \overline{F}_1(\tfrac{y}{A_{11}}) \diff y\right] + \frac{1}{c_2^\ast} \EE\left[\int_{b_2 u}^\infty \overline{F}_1(\tfrac{y}{A_{21}}) \diff y\right] - C_\vee \overline{F}_1(u) \right)  \right.\\
		&\qquad +\left. \lambda_2 \left(\frac{1}{c_1^\ast} \EE\left[\int_{b_1 u}^\infty \overline{F}_2(\tfrac{y}{A_{12}}) \diff y\right]  + \frac{1}{c_2^\ast} \EE\left[\int_{b_2 u}^\infty \overline{F}_2(\tfrac{y}{A_{22}}) \diff y\right]  - C_\vee \overline{F}_2(u) \right)\right)
		\end{align*}
		and hence \eqref{eq-andasymptotic} by substituting $v=\tfrac{y-b_i u}{c_i^\ast}$.
		If \eqref{eq_notAsymptoticEquivalent} fails, then the statement follows in analogy to the proof of Proposition \ref{Proposition_Subexponential_Risk_And}.
	\end{proof}
	
		\subsection{Proofs for Section \ref{S4}}\label{S4c}
	
	\begin{proof}[Proof of Lemma \ref{lem-singleruinlight}]
		We take up the notation used in the proof of Lemma \ref{lem-singleruinasymptotic} and denote the jumps of the resulting one-dimensional risk processes by $\{Y_{i,k},k\in\NN\}$, $i=1,2$. 
		Then the given bound for $\Psi_i(u)$ follows from \cite[Thm. IV.5.2]{asmussenalbrecher} with $\kappa_i>0$ such that $c_i\kappa_i = \lambda (\varphi_{Y_i}(\kappa_i)-1)$. (Note that in \cite{asmussenalbrecher} the constants $c$ and $\lambda$ are combined as $\beta=\lambda/c$.) But since by conditioning 
		\begin{align*}
		\varphi_{Y_i}(y) &=  \mathbb{E}\big[e^{y (B_{11}A_{i1}X_1 + B_{22} A_{i2}X_2)} \big] 
		= \frac{\lambda_1}{\lambda} \mathbb{E}\left[e^{y A_{i1}X_1 }\right] +\frac{\lambda_2}{\lambda} \mathbb{E}\left[e^{y A_{i2}X_2} \right] \\ 
		&= \frac{\lambda_1}{\lambda} \mathbb{E}\left[\varphi_{X_1}(y A_{i1}) \right] +\frac{\lambda_2}{\lambda} \mathbb{E}\left[\varphi_{X_2}(y A_{i2}) \right], \quad i=1,2,
		\end{align*}
		this is equivalent to \eqref{eq_ThLightCond}. \\
		Further by \cite[Thm. IV.5.3]{asmussenalbrecher} it holds
		$$\lim_{u\to\infty} e^{\kappa_i u} \Psi_i(u) = \frac{c_i-\lambda\EE[Y_i]}{\lambda \varphi_{Y_i}'(\kappa_i)-c_i},$$
		with $\EE[Y_i]$ as given in \eqref{eq-ewertY} and
		\begin{align*}
		\varphi_{Y_i}'(y)&= \frac{\lambda_1}{\lambda} \frac{\diff}{\diff y} \mathbb{E}\left[e^{y A_{i1}X_1 }\right] +\frac{\lambda_2}{\lambda} \frac{\diff}{\diff y}\mathbb{E}\left[e^{y A_{i2}X_2} \right]=  \frac{\lambda_1}{\lambda}  \varphi'_{A_{i1}X_i}(y) +\frac{\lambda_2}{\lambda} \varphi'_{A_{i2}X_2}(y),
		\end{align*}
		where, again by conditioning, 
		\begin{align*}
		\varphi'_{A_{ij}X_j}(y)	&= \mathbb{E}\left[A_{ij}X_j e^{y A_{ij}X_j }\right] = \EE\left[ \mathbb{E}\left[A_{ij}X_j e^{y A_{ij}X_j }|A_{ij} \right] \right]  = \EE\left[ A_{ij} \frac{\partial}{\partial(yA_{ij})} \EE[ e^{yA_{ij}X_j}|A_{ij}]\right]\\
		&=  \mathbb{E}\left[A_{ij} \varphi_{X_j}'(y A_{ij})\right], 
		\end{align*}
		which yields the given asymptotics.	
	\end{proof}

\begin{proof}[Proof of Theorem \ref{Theorem_Light_Asymptotics}]
	Recall from Section \ref{S2b} that 
	$$R_i(t)= \sum_{k=1}^{N(t)} \left((A_{i1})_k (B_{11})_k X_{1,k} + (A_{i2})_k (B_{22})_k X_{2,k}\right)	- tc_i,$$
	such that the joint cumulant exponent of the two-dimensional Lévy process $(-R_1(t_1),-R_2(t_2))$ can be determined via conditioning first on $(\bB_k)_{k\in\NN}$, then on the components of $\bA$, as 
	\begin{align*}
	k(t_1,t_2)&= \log \EE[\exp(-t_1 R_1(1) - t_2 R_2(1))] \\
	&= \log \EE\left[ \exp\left(- \sum_{k=1}^{N(1)} \Big( \big(t_1 (A_{11})_k (B_{11})_k  + t_2 (1-(A_{11})_k) (B_{11})_k\big)  X_{1,k} \right. \right. \\
	&  \qquad \qquad \left.\left.+ \left( t_1(A_{12})_k (B_{22})_k + t_2 (1-(A_{12})_k) (B_{22})_k \right)X_{2,k} \Big) + t_1c_1 + t_2 c_2 \right)\right],\\
	&=\log \EE \left[\exp\left( - \sum_{\ell=1}^{N_1(1)} (t_1 (A_{11})_\ell + t_2 (1-(A_{11})_\ell) X_{1,\ell} \right)\right] \\
	& \quad + \log \EE \left[\exp\left( - \sum_{\ell=1}^{N_2(1)} (t_1 (A_{12})_\ell + t_2 (1-(A_{12})_\ell) X_{2,\ell} \right)\right] + t_1c_1 + t_2c_2\\
	&=  \lambda_1\big(\varphi_{(t_1 A_{11} + t_2 (1-A_{11}))X_1}(1)-1\big) +  \lambda_2\big(\varphi_{(t_1 A_{12} + t_2 (1-A_{12}))X_2}(1)-1\big) + t_1c_1 + t_2c_2 \\
	&= \EE\big[\lambda_1 (\varphi_{X_1}(-t_1 A_{11} - t_2 (1-A_{11}))-1)\big] +   \EE\big[ \lambda_2 (\varphi_{X_2}(-t_1 A_{12} - t_2 (1-A_{12}))-1)\big] \\
	&\quad  + t_1c_1 + t_2c_2,
	\end{align*}
	which is by assumption \eqref{eq-lightcondition} well defined on some set $\Xi \supsetneq [0,\infty)^2$. The first two statements thus follow from \cite[Thm. 3]{Avram2009}, as long as there exist $\gamma_1,\gamma_2$, such that $k(-\gamma_1,0)=k(0,-\gamma_2)=0$ and $(-\gamma_1,0),(0,-\gamma_2)\in \Xi^\circ$, the interior of $\Xi$. But since
	\begin{align*}
	k(-x,0)&= \exp\big(\lambda_1 (\EE[\varphi_{ X_1}(x A_{11} )]-1)+ \lambda_2 (\EE[\varphi_{ X_2}(x A_{12})]-1)\big) -xc_1,
	\end{align*}
	we observe that $\gamma_1=\kappa_1$ which exists and is such that $(-\kappa_1,0)\in \Xi^\circ$ by assumption. Likewise we obtain $\gamma_2=\kappa_2$ with $(0,-\kappa_2)\in \Xi^\circ$.  \\
	The last equation now follows directly from the fact that $\Psi_{\wedge,\text{sim}}(u)\leq \Psi_\wedge(u)$. 
\end{proof}

	\small

\end{document}